\documentclass[a4paper]{article}

\usepackage[pages=all, color=black, position={current page.south}, placement=bottom, scale=1, opacity=1, vshift=5mm]{background}

\usepackage[margin=1in]{geometry} 

\usepackage{amsmath}
\usepackage{amsthm}
\usepackage{amssymb}

\usepackage[utf8]{inputenc}
\usepackage{hyperref}

\usepackage{graphicx}
\usepackage{verbatim,xcolor}
\usepackage{bm}
\usepackage{multirow}
\usepackage{algorithmicx}
\usepackage[ruled]{algorithm}
\usepackage{algpseudocode}
\usepackage[symbol]{footmisc}
\usepackage{float, subcaption}

\usepackage[sort&compress,numbers,square]{natbib}

\theoremstyle{plain}
\newtheorem{theorem}{Theorem}[section]

\newtheorem{lemma}[theorem]{Lemma}

\theoremstyle{definition}

\newtheorem{remark}[theorem]{Remark}

\usepackage{graphicx, color}
\graphicspath{{fig/}}
\usepackage{algorithm, algpseudocode}
\usepackage{mathrsfs}

\usepackage{lipsum}
\numberwithin{equation}{section}

\newcommand{\ba}{{\bf a}}
\newcommand{\bb}{{\bf b}}
\newcommand{\bc}{{\bf c}}

\newcommand{\be}{{\bf e}}
\newcommand{\bx}{{\bf x}}
\newcommand{\bn}{{\bf n}}
\newcommand{\bu}{{\bf u}}
\newcommand{\bv}{{\bf v}}

\newcommand{\bw}{{\bf w}}

\newcommand{\bV}{{\bf V}}

\newcommand{\bC}{{\bf C}}
\newcommand{\bD}{{\bf D}}

\newcommand{\bbf}{{\bf f}}

\newcommand{\Pih}{\Pi_h}
\newcommand{\Pz}{\mathcal{P}_0}

\newcommand{\cR}{\mathcal{R}}

\newcommand{\jump}[1]{[ #1 ]}
\newcommand{\avg}[1]{\{ #1\}}
\newcommand{\Hone}{H^1(\Omega)}

\newcommand{\Honezd}{[H^1_0(\Omega)]^d}
 
 \newcommand{\Ltwoz}{L^2_0(\Omega)}
 \newcommand{\Vh}{\bV_h}
 
  \newcommand{\Ch}{\bC_h}
  \newcommand{\Dh}{\bD_h}
\newcommand{\Th}{\mathcal{T}_h}

\newcommand{\Eh}{\mathcal{E}_h}
\newcommand{\Eho}{\mathcal{E}_h^o}
\newcommand{\Ehb}{\mathcal{E}_h^b}

\newcommand{\norm}[1]{\lVert #1\rVert}
\newcommand{\enorm}[1]{\lVert #1\rVert_{\mathcal{E}}}
\newcommand{\snorm}[1]{|#1|}
\newcommand{\trinorm}[1]{{\left\vert\kern-0.25ex\left\vert\kern-0.25ex\left\vert #1 \right\vert\kern-0.25ex\right\vert\kern-0.25ex\right\vert}}

\newcommand{\buh}{\bu_h}
\newcommand{\ph}{p_h}

\title{A uniform and pressure-robust enriched Galerkin method for the Brinkman equations\thanks{Lin Mu's work was supported in part by the U.S. National Science Foundation grant DMS-2309557.}
}
\author{Seulip Lee,\thanks{Department of Mathematics, University of Georgia, Athens, GA 30602 (\texttt{seulip.lee@uga.edu})}
\and Lin Mu,\thanks{Department of Mathematics, University of Georgia, Athens, GA 30602 (\texttt{linmu@uga.edu})}}

\date{
}

\begin{document}
	\maketitle
	
	\begin{abstract}
		This paper presents a pressure-robust enriched Galerkin (EG) method for the Brinkman equations with minimal degrees of freedom based on EG velocity and pressure spaces. The velocity space consists of linear Lagrange polynomials enriched by a discontinuous, piecewise linear, and mean-zero vector function per element, while piecewise constant functions approximate the pressure. We derive, analyze, and compare two EG methods in this paper: standard and robust methods. The standard method requires a mesh size to be less than a viscous parameter to produce stable and accurate velocity solutions, which is impractical in the Darcy regime. Therefore, we propose the pressure-robust method by utilizing a velocity reconstruction operator and replacing EG velocity functions with a reconstructed velocity. The robust method yields error estimates independent of a pressure term and shows uniform performance from the Stokes to Darcy regimes, preserving minimal degrees of freedom. We prove well-posedness and error estimates for both the standard and robust EG methods. We finally confirm theoretical results through numerical experiments with two- and three-dimensional examples and compare the methods' performance to support the need for the robust method.
		
	\end{abstract}


\section{Introduction}
\label{sec:intro}
We consider the stationary Brinkman equations in a bounded domain $\Omega\subset\mathbb{R}^d$ for $d=2,3$ with simply connected Lipschitz boundary $\partial\Omega$: Find fluid velocity $\bu:\Omega\rightarrow \mathbb{R}^d$ and pressure $p:\Omega\rightarrow\mathbb{R}$ such that
\begin{subequations}\label{sys:governing}
\begin{alignat}{2}
-\mu\Delta \mathbf{u} + \frac{\mu}{K} \mathbf{u} + \nabla p &= \mathbf{f}\quad\text{ in }\Omega,\label{eqn: governing1}\\
\nabla\cdot\mathbf{u} &= 0\quad\text{ in }\Omega,\label{eqn: governing2}\\
\bu &=0\quad\text{ on } \partial\Omega\label{eqn: governing3},
\end{alignat}
\end{subequations}
where $\mu$ is fluid viscosity, $K$ is media permeability, and $\bbf$ is a given body force.
The Brinkman equations describe fluid flow in porous media characterized by interconnected pores that allow for the flow of fluids, considering both the viscous forces within the fluid and the resistance from the porous media. The Brinkman equations provide a mathematical framework for studying and modeling complex phenomena such as groundwater flow, multiphase flow in oil reservoirs, blood flow in biological tissues, and pollutant transport in porous media.
In this paper, for simplicity, we consider the scaled Brinkman equations
\begin{subequations}\label{sys: scaled_Brinkman}
\begin{alignat}{2}
-\nu\Delta \mathbf{u} + \mathbf{u} + \nabla p &= \mathbf{f}\quad\text{ in }\Omega,\label{eqn: scaled_Brinkman1}\\
\nabla\cdot\mathbf{u} &= 0\quad\text{ in }\Omega, \label{eqn: scaled_Brinkman2}\\
\bu &=0\quad\text{ on } \partial\Omega \label{eqn: scaled_Brinkman3},
\end{alignat}
\end{subequations}
where $\nu\in[0,1]$ is a viscous parameter.
Mathematically, the Brinkman equations can be seen as a combination of the Stokes and Darcy equations.
When $\nu\rightarrow1$, the Brinkman equations approach a Stokes regime affected by the viscous forces, so standard mixed formulations require the $H^1$-conformity for velocity.
On the other hand, since the Darcy model becomes more prominent as $\nu\rightarrow 0$, finite-dimensional spaces for velocity are forced to satisfy the $H(\text{div})$-conformity.
This compatibility in velocity spaces makes it challenging to construct robust numerical solvers for the Brinkman equations in both the Stokes and Darcy regimes.
The numerical tests in \cite{hannukainen2011computations,mardal2002robust} show that standard mixed methods with well-known inf-sup stables Stokes elements, such as MINI and Taylor-Hood elements,  produce suboptimal orders of convergence in the Darcy regime.
Moreover, with piecewise constant approximations for pressure, the standard methods' velocity errors do not converge in the Darcy regime, while mesh size decreases.
On the other hand, Darcy elements such as Raviart-Thomas and Brezzi-Douglas-Marini do not work for the Stokes domain because they do not satisfy the $H^1$-conformity.
Therefore, the development of robust numerical solvers for the Brinkman equations has had considerable attention.

There have been three major categories in developing robust numerical methods for the Brinkman equations. The first category considers Stokes/Darcy elements and adds stabilization (or penalty) terms or degrees of freedom to impose normal/tangential continuity, respectively. This approach allows Stokes elements to cover the Darcy regime \cite{burman2005stabilized,xie2008uniformly} or $H(\text{div})$-conforming finite elements to be extended to the Stokes regime \cite{konno2010non,konno2012numerical,xie2008uniformly,johnny2012family}. Also, the stabilized method in \cite{burman2008pressure} coarsens a pressure space and applies a stabilization term on pressure, while the robust method in \cite{mardal2002robust} uses an enlarged velocity space. The second approach is to introduce another meaningful unknown and define its suitable formulation and finite-dimensional space, such as velocity gradient \cite{zhao2020new,fu2019parameter,gatica2015analysis,howell2016dual}, vorticity \cite{correa2009unified,anaya2015augmented,vassilevski2014mixed}, and Lagrange multipliers at elements' boundaries \cite{jia2021hybridized}. The third direction is the development of a velocity reconstruction operator, first introduced in \cite{linke2012divergence}, mapping Stokes elements into an $H(\text{div})$-conforming space. In a discrete problem for the Brinkman equations, reconstructed velocity functions replace Stokes elements in the Darcy term and the test function on the right-hand side. This idea has been adopted for a uniformly robust weak Galerkin method for the Brinkman equations \cite{mu2020uniformly}, which inspires our work because of its simplicity in modification.

Our research focuses on developing a robust numerical method for the Brinkman equations with minimal degrees of freedom. The enriched Galerkin (EG) velocity and pressure spaces have been proposed by \cite{YiEtAl22Stokes} for solving the Stokes equations with minimal degrees of freedom. The velocity space consists of linear Lagrange polynomials enriched by a discontinuous, piecewise linear, and mean-zero vector function per element, while piecewise constant functions approximate the pressure. More precisely, a velocity function $\bv=\bv^C+\bv^D$ consists of a continuous linear Lagrange polynomial $\bv^C$ and a discontinuous piecewise linear enrichment function $\bv^D$, so interior penalty discontinuous Galerkin (IPDG) formulations are adopted to remedy the discontinuity of $\bv^D$. These velocity and pressure spaces satisfy the inf-sup condition for the Stokes equations, so they are stable Stokes elements.
We first observe a standard EG method derived from adding the Darcy term $(\bu,\bv)_\Omega$ to the Stokes discrete problem in \cite{YiEtAl22Stokes}.
Our numerical analysis and experiments show that the standard EG method provides stable solutions and convergent errors for the Brinkman equations if a mesh size satisfies the condition $h<\sqrt{\nu}$ that is impractical in the Darcy regime ($\nu\rightarrow0$). Hence, inspired by \cite{mu2020uniformly}, we use the velocity reconstruction operator \cite{HuLeeMuYi} mapping the EG velocity to the first-order Brezzi-Douglas-Marini space, whose consequent action is preserving the continuous component $\bv^C$ and mapping only the discontinuous component $\bv^D$ to the lowest-order Raviart-Thomas space. Then, we replace the EG velocity in the Darcy term and the test function on the right-hand side with the reconstructed linear $H(\text{div})$-conforming velocity.
Therefore, with this simple modification, our resulting EG method yields pressure-robust error estimates and shows uniform performance from the Stokes to Darcy regime without any restriction in a mesh size, which is verified by our numerical analysis and experiments. Through two- and three-dimensional examples, we compare the numerical performance of our robust EG and the standard EG methods with the viscous parameter $\nu$ and mesh size $h$. The numerical results demonstrate why the standard EG method is not suitable for the Brinkman equations in the Darcy regime and show that the robust EG method has uniform performance in solving the Brinkman equations.

The remaining sections of this paper are structured as follows:
Some important notations and definitions are introduced in Section~\ref{sec:prelim}.
In Section~\ref{sec:egmethods}, we introduce the standard and robust EG methods for the Brinkman equations, recalling the EG velocity and pressure spaces \cite{YiEtAl22Stokes} and the velocity reconstruction operator \cite{HuLeeMuYi}.
We prove the well-posedness and error estimates of the standard EG method in Section~\ref{sec:steg}.
In Section~\ref{sec:ureg}, we show the robust method's well-posedness and error estimates that mathematically verify the uniform performance from the Stokes to Darcy regimes.
Section~\ref{sec:num_exp} validates our theoretical results through numerical
experiments in two and three dimensions. Finally, we summarize our contribution in this paper and discuss
related future research in Section~\ref{sec:conclusion}.


\section{Preliminaries}
\label{sec:prelim}
In this section, we introduce some notations and definitions used in this paper.
For a bounded Lipschitz domain $\mathcal{D}\in\mathbb{R}^d$, where $d=2,3$, we denote the Sobolev space as $H^s(\mathcal{D})$ for a real number $s\geq 0$.
Its norm and seminorm are denoted by $\|\cdot\|_{s,\mathcal{D}}$ and $|\cdot|_{s,\mathcal{D}}$, respectively.
The space $H^0(\mathcal{D})$ coincides with $L^2(\mathcal{D})$, and the $L^2$-inner product is denoted by $(\cdot,\cdot)_\mathcal{D}$.
When $\mathcal{D}=\Omega$, the subscript $\mathcal{D}$ will be omitted.
This notation is generalized to vector- and tensor-valued Sobolev spaces.
The notation $H_0^1(\mathcal{D})$ means the space of $v\in H^1(\mathcal{D})$ such that $v=0$ on $\partial\mathcal{D}$, and $L_0^2(\mathcal{D})$ means the space of $v\in L^2(\mathcal{D})$ such that $(v,1)_\mathcal{D}=0$.
The polynomial spaces of degree less than or equal to $k$ are denoted as $P_k(\mathcal{D})$.
We also introduce the Hilbert space
\begin{equation*}
    H(\text{div},\mathcal{D}):=\{\bv\in [L^2(\mathcal{D})]^d:\text{div}\;\bv\in L^2(\mathcal{D})\}
\end{equation*}
with the norm
\begin{equation*}
    \norm{\bv}_{H(\text{div},\mathcal{D})}^2:=\norm{\bv}_{0,\mathcal{D}}^2+\norm{\text{div}\;\bv}_{0,\mathcal{D}}^2.
\end{equation*}

For discrete setting, we assume that there exists a shape-regular triangulation $\Th$ of $\Omega$ whose elements $T\in \Th$ are triangles in two dimensions and tetrahedrons in three dimensions.
Also, $\Eh$ denotes the collection of all edges/faces in $\Th$, and $\Eh=\Eho\cup\Ehb$, where $\Eho$ is the collection of all the interior edges/faces and $\Ehb$ is that of the boundary edges/faces.
For each element $T\in\Th$, let $h_T$ denote the diameter of $T$ and $\bn_T$ (or $\bn$) denote the outward unit normal vector on $\partial T$.
For each interior edge/face $e\in \Eho$ shared by two adjacent elements $T^+$ and $T^-$, we let $\bn_e$ be the unit normal vector from $T^+$ to $T^-$.
For each $e\in\Ehb$, $\bn_e$ denotes the outward unit normal vector on $\partial \Omega$.
In a triangulation $\Th$, the broken Sobolev space is defined as
\begin{equation*}
    H^s(\Th):=\{v\in L^2(\Omega):v|_T\in H^s(T),\ \forall T\in\Th\},
\end{equation*}
equipped with the norm
\begin{equation*}
    \norm{v}_{s,\Th}:=\left(\sum_{T\in\Th} \norm{v}^2_{s,T}\right)^{1/2}.
\end{equation*}
When $s=0$, the $L^2$-inner product on $\Th$ is denoted by $(\cdot,\cdot)_{\Th}$.
Also, the $L^2$-inner product on $\Eh$ is denoted as $\langle\cdot,\cdot\rangle_{\Eh}$, and the $L^2$-norm on $\Eh$ is defined as
\begin{equation*}
    \norm{v}_{0,\Eh}:=\left(\sum_{e\in\Eh} \norm{v}^2_{0,e}\right)^{1/2}.
\end{equation*}
The piecewise polynomial space corresponding to the broken Sobolev space is defined as
\begin{equation*}
    P_k(\Th) = \{v\in L^2(\Omega): v|_T\in P_k(T),\ \forall T\in\Th\}.
\end{equation*}
In addition, the jump and average of $v$ on $e\in \Eh$ are defined as
\begin{equation*}
    \jump{v}:=\left\{\begin{array}{cl}
        v^+-v^- & \text{on}\ e\in \Eho, \\
        v & \text{on}\ e\in\Ehb,
    \end{array}\right.
    \quad
    \avg{v}:=\left\{\begin{array}{cl}
        (v^++v^-)/2 & \text{on}\ e\in \Eho, \\
        v & \text{on}\ e\in\Ehb,
    \end{array}\right.
\end{equation*}
where $v^{\pm}$ is the trace of $v|_{T^\pm}$ on $e\in \partial T^+\cap\partial T^-$. These definitions are extended to vector- and tensor-valued functions.
We finally introduce the trace inequality that holds for any function $v\in H^1(T)$,
\begin{equation}
    \norm{v}_{0,e}^2\leq C\left(h_T^{-1}\norm{v}_{0,T}^2+h_T\norm{\nabla v}_{0,T}^2\right).\label{eqn: trace}
\end{equation}


\section{Enriched Galerkin Methods for the Brinkman equations}
\label{sec:egmethods}

We first introduce the enriched Galerkin (EG) finite-dimensional velocity and pressure spaces \cite{YiEtAl22Stokes}.
The space of continuous components for velocity is
\begin{equation*}
\Ch = \{\bv^C \in \Honezd : \bv^C|_{T} \in [P_1(T)]^d,\ \forall T \in \Th \}.
\end{equation*}
The space of discontinuous components for velocity is defined as
\begin{equation*}
    \Dh = \{\bv^D \in L^2(\Omega) : \bv^D|_{T} = c (\bx - \bx_T),\ c \in \mathbb{R},\ \forall T \in \Th\},
\end{equation*}
where $\bx_T$ is the barycenter of $T\in\Th$.
Thus, the EG finite-dimensional velocity space is defined as
\begin{equation*}
    \Vh := \Ch \oplus \Dh.
\end{equation*}
We note that any function $\bv\in\Vh$ consists of unique continuous and discontinuous components, $\bv=\bv^C+\bv^D$ for $\bv^C\in\Ch$ and $\bv^D\in\Dh$.
At the same time, the EG pressure space is
\begin{equation*}
    Q_h := \{ q \in \Ltwoz : q|_T \in P_0(T),\ \forall T \in \Th \}.
\end{equation*}
Therefore, we formulate a standard EG method for the Brinkman equations with the pair of the EG spaces $\Vh\times Q_h$ by adding the Darcy term to the Stokes formulation \cite{YiEtAl22Stokes}.
\begin{algorithm}[H]
\caption{Standard enriched Galerkin (\texttt{ST-EG}) method} \label{alg:EG}
Find $( \bu_h, \ph) \in \Vh \times Q_h $ such that
\begin{subequations}\label{sys: eg}
\begin{alignat}{2}
\nu\ba(\buh,\bv) +\bc(\buh,\bv)  - \bb(\bv, \ph) &= (\bbf,  \bv) &&\quad \forall \bv \in\Vh, \label{eqn: eg1}\\
\bb(\buh,q) &= 0 &&\quad \forall q\in Q_h,  \label{eqn: eg2}
\end{alignat}
\end{subequations}
where 
\begin{subequations}\label{sys: bilinear}
\begin{align}
\ba(\bv,\bw) &:= (\nabla \bv,\nabla \bw)_{\Th} -  \langle \avg{\nabla\bv} \bn_e, \jump{\bw} \rangle_{\Eh} \nonumber\\
&\qquad\qquad\qquad\qquad- \langle \avg {\nabla\bw} \bn_e,\jump{\bv}\rangle_{\Eh}  +  \rho_1 \langle h_e^{-1}\jump{\bw},
\jump{\bv}\rangle_{\Eh}, \label{eqn: bia} \\
\bc(\bv,\bw)&:=(\bv,\bw)_{\Th}+\rho_2 \langle h_e\jump{\bw},
\jump{\bv}\rangle_{\Eh},\label{eqn: EGbic}\\
\bb(\bw,q)&:= (\nabla\cdot\bw, q)_{\Th} - \langle \jump{\bw}\cdot\bn_e,\avg{q} \rangle_{\Eh}.  \label{eqn: bib}
\end{align}
\end{subequations}
In this case, $\rho_1 >0$ is an $H^1$-penalty parameter, $\rho_2>0$ is an $L^2$-penalty parameter, and $h_e = |e|^{1/(d-1)}$, where $|e|$ is the length/area of the edge/face $e \in \Eh$. 
\end{algorithm}

\begin{remark}
This algorithm employs interior penalty discontinuous Galerkin (IPDG) formulations because any EG velocity function in $\mathbf{V}_h$ has a discontinuity.
IPDG formulations include two penalty terms scaled by $h_e$ with the penalty parameters $\rho_1$ and $\rho_2$.
The \texttt{ST-EG} method provides reliable numerical solutions in the Stokes regime.
However, this approach may not be effective in solving the Brinkman equations in the Darcy regime because it requires
$H(\text{div})$-conforming discrete velocity functions. Moreover, the \texttt{ST-EG} method's velocity error bounds may depend on a pressure term inversely proportional to $\nu$.
\end{remark}

For this reason, we develop a pressure-robust EG method that produces stable and accurate solutions to Brinkman problems with any value of $\nu\in(0,1]$. 
First, the velocity reconstruction operator \cite{HuLeeMuYi} is defined as $\cR: \Vh \to\mathcal{B}DM_1(\Th)\subset H(\text{div},\Omega)$ such that
\begin{subequations}\label{sys: BDM}
\begin{alignat}{2}
\int_e (\cR \bv) \cdot\bn_e  p_1\  ds & = \int_e \avg{\bv}\cdot\bn_e p_1 \ ds,
 && \quad \forall p_1 \in P_1(e), \ \forall e \in \Eho,  \\
\int_e (\cR \bv) \cdot\bn_e  p_1\  ds & = 0,  && \quad \forall p_1 \in P_1(e), \ \forall e \in \Ehb,
\end{alignat}
\end{subequations}
where $\mathcal{B}DM_1(\Th)$ is the Brezzi-Douglas-Marini space of index 1 on $\Th$.
Then, we propose the pressure-robust EG method as follows.

\begin{algorithm}[H]
\caption{Pressure-robust enriched Galerkin (\texttt{PR-EG}) method} \label{alg:UREG}
Find $( \bu_h, \ph) \in \Vh \times Q_h $ such that
\begin{subequations}\label{sys: robust-eg}
\begin{alignat}{2}
\nu\ba(\buh,\bv) +\tilde{\mathbf{c}}(\buh,\bv)  - \bb(\bv, \ph) &= (\bbf,  \cR\bv) &&\quad \forall \bv \in\Vh, \label{eqn: robust-eg1}\\
\bb(\buh,q) &= 0 &&\quad \forall q\in Q_h,  \label{eqn: robust-eg2}
\end{alignat}
\end{subequations}
where $\ba(\bv,\bw)$ and $\bb(\bv,\bw)$ are defined in \eqref{eqn: bia} and \eqref{eqn: bib}, respectively, and
\begin{equation}
\tilde{\mathbf{c}}(\bv,\bw):=(\cR\bv,\cR\bw)_{\Th}+\rho_2 \langle h_e\jump{\bw},
\jump{\bv}\rangle_{\Eh}.\label{eqn: robust-bic}
\end{equation}
\end{algorithm}

\begin{remark}
    Using the velocity reconstruction operator $\cR$, we force discrete velocity functions in $\Vh$ to be $H(\text{div})$-conforming.
We replace the velocity functions in the bilinear form $(\bv,\bw)_{\Th}$ in \eqref{eqn: EGbic} and the right-hand side with the reconstructed velocity $\cR\bv$.
Thus, the term $(\cR\bv,\cR\bw)_{\Th}$ with the $H(\text{div})$-conforming velocity dominates the \texttt{PR-EG} formulation when $\nu$ approaches to 0 (the Darcy regime).
Moreover, the reconstructed velocity on the right-hand side allows us to obtain error bounds independent of a pressure term inversely proportional to $\nu$.
\end{remark}


\section{Well-Posedness and Error Analysis for ST-EG (Algorithm~\ref{alg:EG})}
\label{sec:steg}

First of all, we introduce the discrete $H^1$-norm in \cite{YiEtAl22Stokes} for all $\bv\in\Vh$,
\begin{equation*}
    \enorm{\bv}^2 := \norm{\nabla \bv}_{0, \Th}^2 + \rho_1 \norm{h_e^{-1/2}  \jump{\bv}}_{0, \Eh}^2,
\end{equation*}
where $\rho_1$ is an $H^1$-penalty parameter. With this norm, the coercivity and continuity results for the bilinear form $\ba(\cdot,\cdot)$ have been proved in \cite{YiEtAl22Stokes}: For a sufficiently large $H^1$-penalty parameter $\rho_1$, there exist positive constants $\kappa_1$ and $\kappa_2$ independent of $\nu$ and $h$ such that
\begin{alignat}{2}
\ba(\bv, \bv) & \ge \kappa_1 \enorm{\bv}^2 && \quad \forall \bv \in \Vh, \label{eqn: ba_coer} \\
|\ba(\bv, \bw)| & \leq \kappa_2 \enorm{\bv}\enorm{\bw} && \quad \forall \bv, \bw \in \Vh. \label{eqn: ba_conti}
\end{alignat}
Then, we define an energy norm for Brinkman problems involving the discrete $H^1$-norm and $L^2$-norm,
\begin{equation*}
    \trinorm{\bv}^2 := \nu\enorm{\bv}^2 + \norm{\bv}_{0}^2 +\rho_2 \norm{h_e^{1/2}  \jump{\bv}}_{0, \Eh}^2.
\end{equation*}
In this case, $\rho_2$ is an $L^2$-penalty parameter that should be sufficiently large for well-posedness, and its simple choice is $\rho_2=\rho_1$.
The following lemma shows an essential norm equivalence between $\trinorm{\cdot}$ and $\enorm{\cdot}$ scaled by $\nu$ and $h$.

\begin{lemma} \label{lemma: norm_equiv}
For given $\nu$ and $h$, we define a positive constant $C_{\textsc{ne}}$ (Norm Equivalence) as
\begin{equation*}
C_{\textsc{ne}}:=C\sqrt{\nu+h^2\left(\frac{\rho_2}{\rho_1}+1\right)},
\end{equation*}
where $C$ is a generic positive constant independent of $\nu$ and $h$.
Then, the following norm equivalence holds: For any $\bv\in\Vh$, we have
\begin{equation}\label{eqn: norm_equiv}
\sqrt{\nu}\enorm{\bv}\leq\sqrt{\nu+c_1 h^2}\enorm{\bv}\leq\trinorm{\bv}\leq C_{\textsc{ne}}\enorm{\bv},
\end{equation}
for some small $0<c_1<1$. Moreover, the constant $C_{\textsc{ne}}$ is bounded as
\begin{equation}\label{eqn: bound_cnb}
    C_{\textsc{ne}}\leq C( \sqrt{\nu}+h)
\end{equation}
for some generic constant $C>0$.

\end{lemma}
\begin{proof}
We observe each term in the energy norm $$\trinorm{\bv}^2=\nu\enorm{\bv}^2 + \norm{\bv}_{0}^2 +\rho_2 \norm{h_e^{1/2}  \jump{\bv}}_{0, \Eh}^2.$$ 
Since $\left.\bv\right|_T$ is a linear polynomial in the second term, a scaling argument implies
\begin{equation*}
    \norm{\bv}_0\leq Ch\norm{\nabla\bv}_{0,\Th}\leq Ch\enorm{\bv}.
\end{equation*}
For the trace term, we have
\begin{equation*}
    \rho_2 \norm{h_e^{1/2}  \jump{\bv}}_{0, \Eh}^2\leq Ch^2\left(\frac{\rho_2}{\rho_1}\right)\rho_1\norm{h_e^{-1/2}  \jump{\bv}}_{0, \Eh}^2\leq Ch^2\left(\frac{\rho_2}{\rho_1}\right)\enorm{\bv}^2.
\end{equation*}
Thus, we obtain
\begin{equation*}
    \trinorm{\bv}^2\leq C\left(\nu+h^2\left(\frac{\rho_2}{\rho_1}+1\right)\right)\enorm{\bv}^2.
\end{equation*}
On the other hand, the inverse inequality and the same argument for the trace term lead to
\begin{equation*}
    \enorm{\bv}^2\leq C h^{-2}\left(\norm{\bv}^2_0+\rho_2 \norm{h_e^{1/2}  \jump{\bv}}_{0, \Eh}^2\right),
\end{equation*}
 where $C$ contains $\rho_1/\rho_2$. In this case, we assume $C>1$ and set $c_1=1/C$, so
\begin{equation*}
    (\nu+c_1h^2)\enorm{\bv}^2\leq \trinorm{\bv}^2.
\end{equation*}
\end{proof}

Let us introduce the interpolation operator in \cite{yi2022locking} $\Pih: [H^2(\Omega)]^d \to \Vh$ defined by
\begin{equation*}
\Pi_h\bw=\Pi_h^C\bw+\Pi_h^D\bw,
\end{equation*}
where $\Pi_h^C\bw\in \bC_h$
is the nodal value interpolant of $\bw$ and $\Pi_h^D\bw\in \bD_h$ satisfies 
$(\nabla\cdot\Pi_h^D\bw,1)_T=(\nabla\cdot(\bw - \Pi_h^C \bw), 1)_{T}$ for all $T\in \Th$.
The following interpolation error estimates and stability \cite{yi2022locking} are used throughout our numerical analysis:
\begin{subequations}\label{sys: Pih}
\begin{alignat}{2}
& |\bw - \Pih \bw | _{j,\Th} \leq C h^{m-j} |\bw|_{m},&&\quad 0 \leq j \leq m \leq 2, \quad\forall\bw\in[H^2(\Omega)]^d, \label{eqn: Pih_err} \\
& \enorm{\bw - \Pih \bw} \leq C h \norm{\bw}_2, &&\quad\forall \bw \in [H^2(\Omega)]^d, \label{eqn: Pih_energy_err}
\\
& \enorm{\Pih \bw} \leq C \snorm{\bw}_1,
&&\quad\forall \bw \in \Honezd. \label{eqn: Pih_stability}
\end{alignat}
\end{subequations}
For the pressure, we introduce the local $L^2$-projection $\mathcal{P}_0: \Hone \to Q_h$ such that $(q - \Pz q, 1)_T = 0$ for all $T\in\Th$. Its interpolation error estimate is given as,
\begin{equation}
  \norm{ q - 
  \Pz q}_0 \leq C h \norm{q}_1,  \quad \forall q \in H^1(\Omega).  \label{eqn: P_err}  
\end{equation}

\subsection{Well-posedness}
We first prove the coercivity and continuity results concerning the energy norm $\trinorm{\cdot}$.
\begin{lemma}\label{lemma: coer_conti}
For any $\bv,\bw\in \mathbf{V}_h$, we have the coercivity and continuity results:
\begin{align}
    \nu\ba(\bv,\bv)+\bc(\bv,\bv)&\geq K_1\trinorm{\bv}^2,\label{eqn: ac_coer}\\
    |\nu\ba(\bv,\bw)+\bc(\bv,\bw)|&\leq K_2\trinorm{\bv}\trinorm{\bw},\label{eqn: ac_conti}
\end{align}
where $K_1=\min(\kappa_1,1)$ and $K_2=\max(\kappa_2,1)$.
\end{lemma}
\begin{proof}
If we observe the bilinear forms $\ba(\cdot,\cdot)$ and $\bc(\cdot,\cdot)$ and use the coercivity \eqref{eqn: ba_coer}, then we have
\begin{align*}
\nu\ba(\bv,\bv)+\bc(\bv,\bv)&\geq \kappa_1\nu\enorm{\bv}^2+\norm{\bv}_{0}^2 +\rho_2 \norm{h_e^{1/2}  \jump{\bv}}_{0, \Eh}^2\\
&\geq \min(\kappa_1,1)\trinorm{\bv}^2.
\end{align*}
Moreover, it follows from the Cauchy-Schwarz inequality and the continuity \eqref{eqn: ba_conti} that
\begin{align*}
|\nu\ba(\bv,\bw)+\bc(\bv,\bw)|& \leq \kappa_2\nu\enorm{\bv}\enorm{\bw}+\norm{\bv}_{0}\norm{\bw}_{0}\\
    &\qquad\qquad + \left(\sqrt{\rho_2}\norm{h_e^{1/2}\jump{\bv}}_{0,\Eh}\right)\left(\sqrt{\rho_2}\norm{h_e^{1/2}\jump{\bw}}_{0,\Eh}\right)\\
    &\leq \max(\kappa_2,1)\trinorm{\bv}\trinorm{\bw}.
\end{align*}
\end{proof}

Next, we prove the discrete inf-sup condition for the problem \eqref{sys: eg} in Algorithm~\ref{alg:EG}.
\begin{lemma} \label{thm: dis inf-sup}
Assume that the penalty parameters $\rho_1$ and $\rho_2$ are sufficiently large.
Then, there exists a positive constant $C_{1}:=C_{\textsc{is}}/C_{\textsc{ne}}$ such that
\begin{equation}
\sup_{\bv\in\Vh}\frac{\bb(\bv,q)}{\trinorm{\bv}}\geq C_{1}\norm{q}_0,\quad \forall q\in Q_h,\label{eqn: inf-sup}
\end{equation}
where $C_{\textsc{is}}>0$ (Inf-Sup), independent of $\nu$ and $h$, is the constant for the inf-sup condition for $\enorm{\cdot}$ in \cite{YiEtAl22Stokes}.
\end{lemma}
\begin{proof}
It follows from the discrete inf-sup condition in \cite{YiEtAl22Stokes} and the upper bound of $\trinorm{\bv}$ in \eqref{eqn: norm_equiv} that
\begin{equation*}
    C_{\textsc{is}}\norm{q}_0\leq \sup_{\bv\in\Vh}\frac{\bb(\bv,q)}{\enorm{\bv}}\leq C_{\textsc{ne}}\sup_{\bv\in\Vh}\frac{\bb(\bv,q)}{\trinorm{\bv}}.
\end{equation*}
\end{proof}

Furthermore, Lemma~\ref{lemma: norm_equiv} yields the continuity of $\bb(\cdot,\cdot)$ with $\trinorm{\bv}$.
\begin{lemma}\label{lemma: conti_b}
For any $\bv\in\Vh$ and $q\in Q_h$, there exists a positive constant $C$ independent of $\nu$ and $h$ such that
\begin{equation}
    |\bb(\bv,q)|\leq \frac{C}{\sqrt{\nu+c_1 h^2}}\norm{q}_0\trinorm{\bv}.\label{eqn: contib}
\end{equation}
\end{lemma}
\begin{proof}
It follows from
the continuity of $\bb(\cdot,\cdot)$ in \cite{YiEtAl22Stokes} and
the upper bound of $\enorm{\bv}$ in \eqref{eqn: norm_equiv} that
    \begin{equation*}
    |\bb(\bv,q)|\leq C\norm{q}_0\enorm{\bv}\leq \frac{C}{\sqrt{\nu+c_1 h^2}}\norm{q}_0\trinorm{\bv}.
\end{equation*}

\end{proof}

Thus, we obtain the well-posedness of the \texttt{ST-EG} method in Algorithm~\ref{alg:EG}.
\begin{theorem}\label{thm: well_posed}
There exists a unique solution $(\buh,\ph)\in \Vh\times Q_h$ to the \texttt{ST-EG} method. 
\end{theorem}
\begin{proof}
It suffices to show that $\bu_h=\mathbf{0}$ and $p_h=0$ when $\bbf=\mathbf{0}$ because $\Vh$ and $Q_h$ are finite-dimensional spaces.
Choosing $\bv=\bu_h$ in \eqref{eqn: eg1} and $q=p_h$ in \eqref{eqn: eg2} and adding the two equations imply $\nu\ba(\bu_h,\bu_h)+\bc(\bu_h,\bu_h)=0$.
Hence, $\trinorm{\bu_h}=0$ by \eqref{eqn: ac_coer}, so $\bu_h=\mathbf{0}$.
If $\bu_h=\mathbf{0}$ in \eqref{sys: eg}, then $\bb(\bv,p_h)=0$ for all $\bv\in\Vh$. Therefore, the inf-sup condition \eqref{eqn: inf-sup} yields $\norm{p_h}_0=0$, so $p_h=0$.
\end{proof}

\subsection{Error estimates}
Let $(\bu,p)\in [H_0^1(\Omega)\cap H^2(\Omega)]^d\times [L_0^2(\Omega)\cap H^1(\Omega)]$ be the solution to \eqref{eqn: governing1}-\eqref{eqn: governing3}.
We define the error functions used in the error estimates
\begin{equation*}
    \bm{\chi}_h:=\bu-\Pi_h\bu,\quad\mathbf{e}_h:=\Pi_h\bu-\bu_h,\quad\xi_h:=p-\Pz p,\quad\epsilon_h:=\Pz p-p_h.
\end{equation*}

First, we derive error equations in the following lemma.

\begin{lemma}\label{lemma: eg_erreqn}
For any $\bv\in\Vh$ and $q\in Q_h$, we have
\begin{subequations}\label{sys: eg_erreqn}
\begin{alignat}{2}
\nu\ba(\be_h,\bv)+\bc(\be_h,\bv)-\bb(\bv,\epsilon_h)&=l_1(\bu,\bv)+l_2(\bu,\bv)+\mathbf{s}(\Pi_h\bu,\bv)+\bb(\bv,\xi_h),\label{eqn: eg_erreqn1}\\
\bb(\be_h,q)&=-\bb(\bm{\chi}_h,q),\label{eqn: eg_erreqn2}
\end{alignat}
\end{subequations}
where the supplemental bilinear forms are defined as follows:
\begin{align*}
    &l_1(\bu,\bv):=\nu\ba(\Pi_h\bu-\bu,\bv),\\
    &l_2(\bu,\bv):=(\Pi_h\bu-\bu, \bv)_{\Th},\\
    &\mathbf{s}(\Pi_h\bu,\bv):=\rho_2\langle h_e\jump{\Pi_h\bu},\jump{\bv}\rangle_{\Eh}.
\end{align*}
\end{lemma}

\begin{proof}
We have $-(\Delta\bu,\bv)
_{\Th}=\ba(\bu,\bv)$ for any $\bv\in\Vh$ from \cite{YiEtAl22Stokes}, which implies that
\begin{align*}
    -\nu(\Delta\bu,\bv)_{\Th}
    =\nu\ba(\Pi_h\bu,\bv)-\nu\ba(\Pi_h\bu-\bu,\bv).
\end{align*}
The definition of $\bc(\cdot,\cdot)$ also gives
\begin{align*}
    (\bu,\bv)_{\Th}=\bc(\Pi_h\bu,\bv)-(\Pi_h\bu-\bu, \bv)_{\Th}-\rho_2\langle h_e\jump{\Pi_h\bu},\jump{\bv}\rangle_{\Eh},
\end{align*}
and integration by parts and continuity of $p$ lead to
\begin{equation*}
    (\nabla p,\bv)_{\Th} = \sum_{T\in\Th}\langle p,\bv\cdot\bn\rangle_{\partial T} -(p,\nabla\cdot \bv)_T= -\bb(\bv,p).
\end{equation*}
Thus, the equation \eqref{eqn: governing1} imposes
\begin{equation*}
    \nu\ba(\Pi_h\bu,\bv)+\bc(\Pi_h\bu,\bv)-\bb(\bv,p)=(\bbf,\bv)+l_1(\bu,\bv)+l_2(\bu,\bv)+\mathbf{s}(\Pi_h\bu,\bv).
\end{equation*}
By comparing this equation with \eqref{eqn: eg1} in the \texttt{ST-EG} method, we arrive at
\begin{equation*}
    \nu\ba(\be_h,\bv)+\bc(\be_h,\bv)-\bb(\bv,\epsilon_h)=l_1(\bu,\bv)+l_2(\bu,\bv)+\mathbf{s}(\Pi_h\bu,\bv)+\bb(\bv,\xi_h).
\end{equation*}
Moreover, it follows from the continuity of $\bu$ and \eqref{eqn: eg2} that
\begin{equation*}
    (\nabla\cdot\bu,q)_{\Th}=\bb(\bu,q)=0=\bb(\bu_h,q),
\end{equation*}
which implies \eqref{eqn: eg_erreqn2}.

\end{proof}

In what follows, we prove estimates for the supplemental bilinear forms in Lemma~\ref{lemma: eg_erreqn}.
\begin{lemma}\label{lemma: eg_supp_est}
Assume that $\bw\in[H^2(\Omega)]^d$ and $\bv\in \Vh$. Then, we have
\begin{subequations}\label{sys: eg_supp_est}
\begin{alignat}{2}
& \left|l_1(\bw,\bv)\right|\leq C\sqrt{\nu} h\norm{\bw}_2\trinorm{\bv},\label{eqn: eg_suppest1} \\
& \left|l_2(\bw,\bv)\right|\leq C h^2\norm{\bw}_2\trinorm{\bv},\label{eqn: eg_suppest2} \\
& \left|\mathbf{s}(\Pi_h\bw,\bv)\right|\leq Ch^2\norm{\bw}_2\trinorm{\bv},\label{eqn: eg_suppest4}
\end{alignat}
\end{subequations}
where $C$ is a generic positive constant independent of $\nu$ and $h$ and may vary in each case.
\end{lemma}

\begin{proof}
It follows from \eqref{eqn: ba_conti}, \eqref{eqn: Pih_energy_err}, and \eqref{eqn: norm_equiv} that
\begin{align*}
    \left|l_1(\bw,\bv)\right|&=\left|\nu\ba(\Pi_h\bw-\bw,\bv)\right|\\
    &\leq\nu \kappa_2\enorm{\Pi_h\bw-\bw}\enorm{\bv}\\
    &\leq C\nu h \norm{\bw}_2\enorm{\bv}\\
    &\leq C\sqrt{\nu}h\norm{\bw}_2\trinorm{\bv}.
\end{align*}
Using the Cauchy-Schwarz inequality and \eqref{eqn: Pih_err},
we get the following upper bounds
\begin{align*}
    \left|l_2(\bw,\bv)\right|&=\left|(\Pi_h\bw-\bw,\bv)_{\Th}\right|\\
    &\leq\norm{\Pi_h\bw-\bw}_0\norm{\bv}_0\\
    &\leq Ch^2|\bw|_2\trinorm{\bv}.
\end{align*}
Finally, the Cauchy-Schwarz inequality, trace inequality \eqref{eqn: trace}, and \eqref{eqn: Pih_err} imply
\begin{align*}
    \left|\mathbf{s}(\Pi_h\bw,\bv)\right|&=\left|\rho_2\langle h_e\jump{\Pi_h\bw},\jump{\bv}\rangle_{\Eh}\right|\\
    &=\left|\rho_2\langle h_e\jump{\Pi_h\bw-\bw},\jump{\bv}\rangle_{\Eh}\right|\\
    &\leq \rho_2\norm{h_e^{1/2}\jump{\Pi_h\bw-\bw}}_{0,\Eh}\norm{h_e^{1/2}\jump{\bv}}_{0,\Eh}\\
    &\leq\norm{h_e^{1/2}\jump{\Pi_h\bw-\bw}}_{0,\Eh}\trinorm{\bv}\\
    &\leq Ch^2|\bw|_2\trinorm{\bv}.
\end{align*}
\end{proof}

In addition, we expand the continuity of $\bb(\cdot,\cdot)$ in \cite{YiEtAl22Stokes} to be relevant to the error equations \eqref{sys: eg_erreqn} because $\bm{\chi}_h=\bu-\Pi_h\bu\not\in\mathbf{V}_h$ and $\xi_h=p-\Pz p\not\in Q_h$.
\begin{lemma} For any $\bv\in\Vh$ and $q\in Q_h$, we have
    \begin{subequations}\label{sys: supp_conti}
\begin{alignat}{2}
& |\bb(\bv,\xi_h)|\leq Ch \norm{p}_1\enorm{\bv},\label{eqn: supp_conti1} \\
& |\bb(\bm{\chi}_h,q)|\leq Ch\norm{q}_0\norm{\bu}_2,\label{eqn: supp_conti2}
\end{alignat}
\end{subequations}
where $C$ is a generic positive constant independent of $\nu$ and $h$ and may vary in each case.
\end{lemma}
\begin{proof}
First, we use the Cauchy-Schwarz inequality to get
    \begin{align*}
        |\bb(\bv,\xi_h)|&=|(\nabla\cdot\bv,\xi_h)_{\Th}-\langle \jump{\bv}\cdot\bn_e,\avg{\xi_h}\rangle_{\Eh}|\\
        &\leq C\left(\norm{\nabla \bv}_{0,\Th}\norm{\xi_h}_0+\norm{h_e^{-1/2}\jump{\bv}}_{0,\Eh}\norm{h_e^{1/2}\avg{\xi_h}}_{0,\Eh}\right).
    \end{align*}
Then, the trace term is bounded by using the trace inequality \eqref{eqn: trace} and interpolation error estimate \eqref{eqn: P_err},
    \begin{equation*}
        \norm{h_e^{1/2}\avg{\xi_h}}_{0,\Eh}^2\leq C\left(\norm{\xi_h}_0^2+h^2\norm{\nabla\xi_h}_{0,\Th}^2\right)\leq Ch^2\norm{p}_1^2
    \end{equation*}
because $\nabla\xi_h=\nabla(p-\Pz p)=\nabla p$.
Hence, the definition of the discrete $H^1$-norm and estimate \eqref{eqn: P_err} imply
    \begin{equation*}
        |\bb(\bv,\xi_h)|\leq Ch\norm{p}_1\enorm{\bv}.
    \end{equation*}
Similarly, it follows from the Cauchy-Schwarz inequality, trace inequality \eqref{eqn: trace}, and \eqref{eqn: Pih_energy_err} that
    \begin{align*}
        |\bb(\bm{\chi}_h,q)|&\leq C\left(\norm{\nabla \bm{\chi}_h}_{0,\Th}\norm{q}_0+\norm{h_e^{-1/2}\jump{\bm{\chi}_h}}_{0,\Eh}\norm{h_e^{1/2}\avg{q}}_{0,\Eh}\right)\\
        &\leq C\norm{q}_0\enorm{\bm{\chi}_h}\leq Ch\norm{q}_0\norm{\bu}_2.
    \end{align*}
\end{proof}

Therefore, we show error estimates of the \texttt{ST-EG} method in Algorithm~\ref{alg:EG} for the Brinkman equations.

\begin{theorem}\label{thm: eg_error_est1}
Let $(\bu,p)\in [H_0^1(\Omega)\cap H^2(\Omega)]^d\times [L_0^2(\Omega)\cap H^1(\Omega)]$ be the solution to \eqref{eqn: governing1}-\eqref{eqn: governing3}, and $(\buh,p_h)\in \Vh\times Q_h$ be the discrete solution from the \texttt{ST-EG} method. Then, we have the following error estimates
 \begin{align*}
      \trinorm{\Pi_h\bu-\bu_h}&\leq C\left[(\sqrt{\nu}+1)h\norm{\bu}_2 + \left( h+\frac{h}{\sqrt{\nu+c_1 h^2}}\right)\norm{p}_1 \right],\\
      \norm{\Pz p-p_h}_0&\leq C\left[ (\nu+\sqrt{\nu})h\norm{\bu}_2 + (\sqrt{\nu}+1)h\norm{p}_1  \right].
  \end{align*}
\end{theorem}
\begin{proof}
First of all, we apply the continuity results \eqref{eqn: ac_conti}, \eqref{eqn: supp_conti1}, the estimates \eqref{sys: eg_supp_est}, and the norm equivalence \eqref{eqn: norm_equiv} to the error equation \eqref{eqn: eg_erreqn1},
\begin{align*}   \bb(\bv,\epsilon_h)&=\nu\ba(\be_h,\bv)+\bc(\be_h,\bv)-l_1(\bu,\bv)-l_2(\bu,\bv)-\mathbf{s}(\Pi_h\bu,\bv)-\bb(\bv,\xi_h)\\
&\leq C\left(\trinorm{\be_h}+\sqrt{\nu}h\norm{\bu}_2+h^2\norm{\bu}_2+\frac{h}{\sqrt{\nu+c_1 h^2}}\norm{p}_1\right)\trinorm{\bv}.
\end{align*}
Thus, the inf-sup condition \eqref{eqn: inf-sup} with \eqref{eqn: bound_cnb} implies
\begin{equation}\label{eqn: eg_inter_result2}
    \norm{\epsilon_h}_0\leq C(\sqrt{\nu}+h)\left(\trinorm{\be_h}+\sqrt{\nu}h\norm{\bu}_2+h^2\norm{\bu}_2+\frac{h}{\sqrt{\nu+c_1 h^2}}\norm{p}_1\right).
\end{equation}
We choose $\bv=\be_h$ in \eqref{eqn: eg_erreqn1} and $q=\epsilon_h$ in \eqref{eqn: eg_erreqn2} and substitute $\bb(\be_h,\epsilon_h)$ with $-\bb(\bm{\chi}_h,\epsilon_h)$ to obtain
\begin{equation*}
\nu\ba(\be_h,\be_h)+\bc(\be_h,\be_h)=-\bb(\bm{\chi}_h,\epsilon_h)+l_1(\bu,\be_h)+l_2(\bu,\be_h)+\mathbf{s}(\Pi_h\bu,\be_h)+\bb(\be_h,\xi_h).
\end{equation*}
In this case, we estimate the term  $\bb(\bm{\chi}_h,\epsilon_h)$
using \eqref{eqn: supp_conti2},
\begin{equation}\label{eqn: eg_inter_result}
    |\bb(\bm{\chi}_h,\epsilon_h)|\leq Ch\norm{\bu}_2\norm{\epsilon_h}_0.
\end{equation}
The term $\bb(\be_h,\xi_h)$ is estimated by using \eqref{eqn: supp_conti1} and \eqref{eqn: norm_equiv},
\begin{equation}\label{eqn: eg_inter-result3}
    |\bb(\be_h,\xi_h)|\leq Ch\norm{p}_1\enorm{\be_h}\leq C\frac{h}{\sqrt{\nu+c_1h^2}}\norm{p}_1\trinorm{\be_h}.
\end{equation}
Hence, it follows from \eqref{eqn: ac_coer}, \eqref{eqn: eg_inter_result}, \eqref{sys: eg_supp_est}, and \eqref{eqn: eg_inter-result3} that
\begin{equation*}
    \trinorm{\be_h}^2\leq C\left(h\norm{\bu}_2\norm{\epsilon_h}_0 + \sqrt{\nu}h\norm{\bu}_2\trinorm{\be_h}+h^2\norm{\bu}_2\trinorm{\be_h} + \frac{h}{\sqrt{\nu+c_1 h^2}}\norm{p}_1\trinorm{\be_h} \right).
\end{equation*}
We use the estimate \eqref{eqn: eg_inter_result2} and omit high-order terms ($h^3$ or $h^4$) to obtain,
\begin{align*}   
h\norm{\bu}_2\norm{\epsilon_h}_0&\leq C\left( (\sqrt{\nu}+h)h\norm{\bu}_2\trinorm{\be_h} + \nu h^2\norm{\bu}_2^2 + \frac{\sqrt{\nu}+h}{\sqrt{\nu+c_1 h^2}}h^2\norm{\bu}_2\norm{p}_1\right)\\
&\leq C\left( (\sqrt{\nu}+h)h\norm{\bu}_2\trinorm{\be_h} + \nu h^2\norm{\bu}_2^2+ h^2\norm{\bu}_2\norm{p}_1\right)
\end{align*}
because $\sqrt{\nu} +h\leq (\sqrt{2/c_1})\sqrt{\nu+c_1 h^2}$.
If we apply the Young’s inequality to each term with a positive constant $\alpha$, then we have
\begin{align*}
&\sqrt{\nu}h\norm{\bu}_2\trinorm{\be_h}\leq \frac{\nu h^2}{2\alpha}\norm{\bu}_2^2+\frac{\alpha}{2}\trinorm{\be_h}^2,\\
&h^2\norm{\bu}_2\trinorm{\be_h}\leq\frac{h^4}{2\alpha}\norm{\bu}_2^2 + \frac{\alpha}{2}\trinorm{\be_h}^2,\\
&h^2\norm{\bu}_2\norm{p}_1\leq\frac{h^2}{2\alpha}\norm{\bu}_2^2 + \frac{\alpha h^2}{2}\norm{p}_1^2,\\
&\frac{h}{\sqrt{\nu+c_1 h^2}}\norm{p}_1\trinorm{\be_h}\leq \frac{h^2}{2\alpha(\nu+c_1 h^2)}\norm{p}_1^2+\frac{\alpha}{2}\trinorm{\be_h}^2.
\end{align*}
Therefore, a proper $\alpha$ implies
\begin{equation*}
    \trinorm{\be_h}^2\leq C\left[(\nu+1)h^2\norm{\bu}_2^2 + \left( h^2+\frac{h^2}{\nu+c_1 h^2}\right)\norm{p}_1^2 \right],
\end{equation*}
so we finally get
\begin{equation}\label{eqn: inter_velo_esti}
    \trinorm{\be_h}\leq C\left[(\sqrt{\nu}+1)h\norm{\bu}_2 + \left( h+\frac{h}{\sqrt{\nu+c_1 h^2}}\right)\norm{p}_1 \right].
\end{equation}
On the other hand, we observe the intermediate estimate \eqref{eqn: eg_inter_result2} and omit high-order terms ($h^2$ or $h^3$) to show the pressure error estimate,
\begin{equation*}
    \norm{\epsilon_h}_0\leq C\left[(\sqrt{\nu}+h)\trinorm{\be_h}+\nu h\norm{\bu}_2+h\norm{p}_1\right].
\end{equation*}
Thus, we bound $\trinorm{\be_h}$ with the velocity error estimate \eqref{eqn: inter_velo_esti}, so we finally obtain
\begin{equation*}
    \norm{\epsilon_h}_0\leq C\left[ (\nu+\sqrt{\nu})h\norm{\bu}_2 + (\sqrt{\nu}+1)h\norm{p}_1  \right],
\end{equation*}
when omitting $h^2$-terms.
\end{proof}

\begin{remark}
    Theorem~\ref{thm: eg_error_est1} explains that the errors converge in the first order with $h$ under the condition $h<\sqrt{\nu}$ easily satisfied in the Stokes regime.
However, the velocity error in the Darcy regime may not decrease with $h$ due to the pressure term in the velocity error bound, that is, when $\nu\rightarrow 0$,
\begin{equation*}
    \frac{h}{\sqrt{\nu+c_1h^2}}\norm{p}_1\rightarrow\frac{1}{\sqrt{c_1}}\norm{p}_1.
\end{equation*}
We will confirm these theoretical results through numerical experiments.
For this reason, the \texttt{ST-EG} method in Algorithm~\ref{alg:EG} may not be effective in solving the Brinkman equations with small $\nu$, which motivates us to develop and analyze the \texttt{PR-EG} method in Algorithm~\ref{alg:UREG}.
\end{remark}


\section{Well-Posedness and Error Analysis for PR-EG (Algorithm~\ref{alg:UREG})}
\label{sec:ureg}

In this section, we prove well-posedness and error estimates for the \texttt{PR-EG} method in Algorithm~\ref{alg:UREG}.
The error estimates show that the \texttt{PR-EG} method's velocity and pressure errors decrease in the optimal order of convergence in both the Stokes and Darcy regimes, so we expect stable and accurate numerical solutions with any $\nu$ as $h$ decreases.

We first define another energy norm by replacing $\norm{\bv}_0$ with $\norm{\cR\bv}_0$,
\begin{equation*}
    \trinorm{\bv}^2_\mathcal{R} := \nu\enorm{\bv}^2 + \norm{\cR\bv}_{0}^2 +\rho_2 \norm{h_e^{1/2}  \jump{\bv}}_{0, \Eh}^2.
\end{equation*}
We also introduce the interpolation error estimate of the operator $\cR$ in \cite{HuLeeMuYi}.
\begin{lemma} For any $\bv \in \Vh$, there exists a positive constant $C$ independent of $\nu$ and $h$ such that
\begin{equation}\label{eqn: cR-err}
\norm{ \bv - \cR\bv}_{0}\leq Ch\norm{h_e^{-1/2}\jump{\bv}}_{0,\Eh} \leq C h \enorm{\bv}.
\end{equation}
\label{lemma: cR-err}
\end{lemma}
This interpolation error estimate allows to have the norm equivalence between $\trinorm{\bv}_\mathcal{R}$ and $\enorm{\bv}$ scaled by $\nu$ and $h$, similar to Lemma~\ref{lemma: norm_equiv}.
\begin{lemma}\label{lemma: norm_equiv_ur}
For any $\bv\in\Vh$, it holds
\begin{equation}\label{eqn: ur_norm_equiv}
\sqrt{\nu}\enorm{\bv}\leq\sqrt{\nu+c_2 h^2}\enorm{\bv}\leq\trinorm{\bv}_\mathcal{R}\leq C_{\textsc{ne}}\enorm{\bv},
\end{equation}
where $C_{\textsc{ne}}$ is the constant defined in Lemma~\ref{lemma: norm_equiv} and $0<c_2<1$ is a small constant.
\end{lemma}
\begin{proof}
    It suffices to prove that $\norm{\cR\bv}_0\leq Ch\enorm{\bv}$ for the upper bound because $\norm{\bv}_0$ is replaced by $\norm{\cR\bv}_0$ in the norm $\trinorm{\bv}_\mathcal{R}$.
    Indeed, it follows from the triangle inequality, the error estimate \eqref{eqn: cR-err}, and the argument in the proof of Lemma~\ref{lemma: norm_equiv} that
    \begin{equation*}
        \norm{\cR\bv}_0 \leq \norm{\bv}_0 + \norm{\cR\bv-\bv}_0\leq \norm{\bv}_0+Ch\enorm{\bv}\leq Ch\enorm{\bv}.
    \end{equation*}
    Hence, we obtain
    \begin{equation*}
        \trinorm{\bv}_\mathcal{R}^2=\nu\enorm{\bv}^2 + \norm{\cR\bv}_{0}^2 +\rho_2 \norm{h_e^{1/2}  \jump{\bv}}_{0, \Eh}^2\leq C\left(\nu + h^2\left(\frac{\rho_2}{\rho_1}+1\right)\right)\enorm{\bv}^2.
    \end{equation*}
    For the lower bound, we recall the result in Lemma~\ref{lemma: norm_equiv} and apply \eqref{eqn: cR-err} to it,
    \begin{align*}
    \enorm{\bv}^2&\leq C h^{-2}\left(\norm{\bv}_0^2+\rho_2 \norm{h_e^{1/2}  \jump{\bv}}_{0, \Eh}^2\right)\\
    &\leq C h^{-2}\left(\norm{\cR\bv}_0^2+\norm{\cR\bv-\bv}_0^2+\rho_2 \norm{h_e^{1/2}  \jump{\bv}}_{0, \Eh}^2\right)\\
    &\leq Ch^{-2}\left(\norm{\cR\bv}_0^2+ h^2\norm{h_e^{-1/2}\jump{\bv}}_{0,\Eh}^2+\rho_2 \norm{h_e^{1/2}  \jump{\bv}}_{0, \Eh}^2\right)\\
    &=Ch^{-2}\left(\norm{\cR\bv}_0^2+\rho_2 \norm{h_e^{1/2}  \jump{\bv}}_{0, \Eh}^2\right)+C_0\norm{h_e^{-1/2}\jump{\bv}}_{0,\Eh}^2,
\end{align*}
where $C_0$ contains $\rho_1/\rho_2$ but is independent of $\nu$ and $h$.
Then, for a sufficiently large $\rho_1$, we have
\begin{equation*}
    \frac{\rho_1-C_0}{\rho_1}\enorm{\bv}^2\leq Ch^{-2}\left(\norm{\cR\bv}_0^2+\rho_2 \norm{h_e^{1/2}  \jump{\bv}}_{0, \Eh}^2\right).
\end{equation*}
     Therefore, we set $c_2=(\rho_1-C_0)/(C\rho_1)$ and assume $c_2<1$ to have
     \begin{equation*}
         c_2h^2\enorm{\bv}^2\leq \norm{\cR\bv}_0^2+\rho_2\norm{h_e^{1/2}\jump{\bv}}_{0,\Eh}^2,
     \end{equation*}
     which implies
     \begin{equation*}
         (\nu+c_2h^2)\enorm{\bv}\leq\trinorm{\bv}_{\cR}.
     \end{equation*}
\end{proof}

In addition, we prove the norm equivalence between $\trinorm{\bv}$ and $\trinorm{\bv}_{\cR}$ using the results in Lemma~\ref{lemma: norm_equiv}, Lemma~\ref{lemma: cR-err}, and Lemma~\ref{lemma: norm_equiv_ur}.
\begin{lemma}\label{cor: norm_eq_tri}
For any $\bv\in\Vh$, it holds
\begin{equation}\label{eqn: norm_eq_tri}
    c_*\trinorm{\bv}_{\cR}\leq \trinorm{\bv}\leq c^*\trinorm{\bv}_{\cR},
\end{equation}
where $c_*$ and $c^*$ are positive constants independent of $\nu$ and $h$.
\end{lemma}
\begin{proof}
    It follows from the results in Lemma~\ref{lemma: cR-err} and Lemma~\ref{lemma: norm_equiv} that
    \begin{equation*}
        \nu\enorm{\bv}^2+\norm{\cR\bv}_0^2\leq C\left(\nu\enorm{\bv}^2+c_1h^2\enorm{\bv}^2+\norm{\bv}_0^2\right)\leq C\trinorm{\bv}^2.
    \end{equation*}
    Similarly, from Lemma~\ref{lemma: cR-err} and Lemma~\ref{lemma: norm_equiv_ur}, we obtain
    \begin{equation*}
        \nu\enorm{\bv}^2+\norm{\bv}_0^2\leq C\left(\nu\enorm{\bv}^2+c_2h^2\enorm{\bv}^2+\norm{\cR\bv}_0^2\right)\leq C\trinorm{\bv}^2_{\cR}.
    \end{equation*}
\end{proof}

\subsection{Well-posedness}

Most of the results for the well-posedness of the \texttt{PR-EG} method are similar to those of the \texttt{ST-EG} method. Thus, we briefly state and prove the results concerning $\trinorm{\cdot}_\mathcal{R}$ in this subsection.

\begin{lemma}
For any $\bv,\bw\in \mathbf{V}_h$, the coercivity and continuity results hold:
\begin{align}
    \nu\ba(\bv,\bv)+\tilde{\mathbf{c}}(\bv,\bv)&\geq K_1\trinorm{\bv}^2_\mathcal{R},\label{eqn: ur_ac_coer}\\
    |\nu\ba(\bv,\bw)+\tilde{\mathbf{c}}(\bv,\bw)|&\leq K_2\trinorm{\bv}_{\mathcal{R}}\trinorm{\bw}_{\mathcal{R}},\label{eqn: ur_ac_conti}
\end{align}
where $K_1=\min(\kappa_1,1)$ and $K_2=\max(\kappa_2,1)$.
\end{lemma}
\begin{proof}
    The proof is the same as that of Lemma~\ref{lemma: coer_conti}, so we omit the details here.
\end{proof}

\begin{lemma} \label{thm: ur_dis inf-sup}
Assume that the penalty parameters $\rho_1$ and $\rho_2$ are sufficiently large.
Then, we have
\begin{equation}
\sup_{\bv\in\Vh}\frac{\bb(\bv,q)}{\trinorm{\bv}_\mathcal{R}}\geq C_{1}\norm{q}_0,\quad \forall q\in Q_h,\label{eqn: ur_inf-sup}
\end{equation}
for $C_{1}=C_{\textsc{is}}/C_{\textsc{ne}}$ defined in Lemma~\ref{thm: dis inf-sup}.
\end{lemma}
\begin{proof}
Similar to the proof of Lemma~\ref{thm: dis inf-sup}, the discrete inf-sup condition in \cite{YiEtAl22Stokes} and the upper bound of $\trinorm{\bv}_\mathcal{R}$ in \eqref{eqn: ur_norm_equiv} imply
\begin{equation*}
    C_{\textsc{is}}\norm{q}_0\leq \sup_{\bv\in\Vh}\frac{\bb(\bv,q)}{\enorm{\bv}}\leq C_{\textsc{ne}}\sup_{\bv\in\Vh}\frac{\bb(\bv,q)}{\trinorm{\bv}_\mathcal{R}}.
\end{equation*}
\end{proof}

\begin{lemma}
For any $\bv\in\Vh$ and $q\in Q_h$, it holds
\begin{equation}
    |\bb(\bv,q)|\leq \frac{C}{\sqrt{\nu+c_2 h^2}}\norm{q}_0\trinorm{\bv}_\mathcal{R},\label{eqn: ur_contib}
\end{equation}
for a generic positive constant $C$ independent of $\nu$ and $h$.
\end{lemma}
\begin{proof}
Similar to the proof of Lemma~\ref{lemma: conti_b}, this result is proved by the continuity of $\bb(\cdot,\cdot)$ in \cite{YiEtAl22Stokes} and the upper bound of $\enorm{\bv}$ in \eqref{eqn: ur_norm_equiv}.
\end{proof}

Finally, we obtain the well-posedness of the \texttt{PR-EG} method in Algorithm~\ref{alg:UREG}.

\begin{theorem}
There exists a unique solution $(\buh,\ph)\in \Vh\times Q_h$ to the \texttt{PR-EG} method.
\end{theorem}
\begin{proof}
    The proof is the same as Theorem~\ref{thm: well_posed}, so we omit the details here.
\end{proof}

\subsection{Error estimates}

We recall the error functions
\begin{equation*}
    \bm{\chi}_h:=\bu-\Pi_h\bu,\quad\mathbf{e}_h:=\Pi_h\bu-\bu_h,\quad\xi_h:=p-\Pz p,\quad\epsilon_h:=\Pz p-p_h,
\end{equation*}
where $(\bu,p)\in [H_0^1(\Omega)\cap H^2(\Omega)]^d\times [L_0^2(\Omega)\cap H^1(\Omega)]$ is the solution to \eqref{eqn: governing1}-\eqref{eqn: governing3}.
Then, we derive error equations for the \texttt{PR-EG} method.

\begin{lemma}\label{lemma: ur_erreqn}
For any $\bv\in\Vh$ and $q\in Q_h$, we have
\begin{subequations}\label{sys: ur_erreqn}
\begin{alignat}{2}
\nu\ba(\be_h,\bv)+\tilde{\bc}(\be_h,\bv)-\bb(\bv,\epsilon_h)&=l_1(\bu,\bv)+l_3(\bu,\bv)+l_4(\bu,\bv)+\mathbf{s}(\Pi_h\bu,\bv),\label{eqn: ur_erreqn1}\\
\bb(\be_h,q)&=-\bb(\bm{\chi}_h,q),\label{eqn: ur_erreqn2}
\end{alignat}
\end{subequations}
where $l_1(\bu,\bv)$ and $\mathbf{s}(\Pi_h\bu,\bv)$ are defined in Lemma~\ref{lemma: eg_erreqn}, and the other supplemental bilinear forms are defined as follows:
\begin{align*}
    &l_3(\bu,\bv):=\nu(\Delta\bu, \cR\bv-\bv)_{\Th},\\
    &l_4(\bu,\bv):=(\cR\Pi_h\bu-\bu,\cR\bv)_{\Th}.
\end{align*}
\end{lemma}

\begin{proof}
Since $-(\Delta\bu,\bv)
_{\Th}=\ba(\bu,\bv)$ for any $\bv\in\Vh$, we have
\begin{align*}
    -\nu(\Delta\bu,\cR\bv)_{\Th}&=-\nu(\Delta\bu,\bv)_{\Th}-\nu(\Delta\bu,\cR\bv-\bv)_{\Th}\\
    &=\nu\ba(\bu,\bv)-\nu(\Delta\bu,\cR\bv-\bv)_{\Th}\\
    &=\nu\ba(\Pi_h\bu,\bv)-\nu\ba(\Pi_h\bu-\bu,\bv)-\nu(\Delta\bu,\cR\bv-\bv)_{\Th}.
\end{align*}
By the definition of $\tilde{\bc}(\cdot,\cdot)$, we also have
\begin{align*}
    (\bu,\cR\bv)_{\Th}&=(\cR\Pi_h\bu,\cR\bv)_{\Th}-(\cR\Pi_h\bu-\bu,\cR\bv)_{\Th}\\
    &=\tilde{\bc}(\Pi_h\bu,\bv)-(\cR\Pi_h\bu-\bu,\cR\bv)_{\Th}-\rho_2\langle h_e\jump{\Pi_h\bu},\jump{\bv}\rangle_{\Eh}.
\end{align*}
Since $\cR\bv\cdot\bn$ is continuous on $\partial T$ and $\nabla\cdot\cR\bv$ is constant in $T$, integration by parts implies
\begin{equation*}
    (\nabla p,\cR\bv)_{\Th} = -\bb(\bv,\Pz p).
\end{equation*}
Hence, we obtain the following equation from \eqref{eqn: governing1},
\begin{equation*}
    \nu\ba(\Pi_h\bu,\bv)+\tilde{\bc}(\Pi_h\bu,\bv)-\bb(\bv,\Pz p)=(\bbf,\cR\bv)+l_1(\bu,\bv)+l_3(\bu,\bv)+l_4(\bu,\bv)+\mathbf{s}(\Pi_h\bu,\bv).
\end{equation*}
If we compare this equation with \eqref{eqn: robust-eg1} in the \texttt{PR-EG} method, then we arrive at
\begin{equation*}
    \nu\ba(\be_h,\bv)+\tilde{\bc}(\be_h,\bv)-\bb(\bv,\epsilon_h)=l_1(\bu,\bv)+l_3(\bu,\bv)+l_4(\bu,\bv)+\mathbf{s}(\Pi_h\bu,\bv).
\end{equation*}
For the second equation \eqref{eqn: ur_erreqn2}, the continuity of $\bu$ and \eqref{eqn: robust-eg2} in the \texttt{PR-EG} method lead us to
\begin{equation*}
    (\nabla\cdot\bu,q)_{\Th}=\bb(\bu,q)=0=\bb(\bu_h,q).
\end{equation*}

\end{proof}

We present estimates for the supplementary bilinear forms used in Lemma~\ref{lemma: ur_erreqn}.

\begin{lemma}\label{lemma: ur_supplement estimates}
Assume that $\bw\in[H^2(\Omega)]^d$ and $\bv\in \Vh$. Then, we have
\begin{subequations}\label{sys: ur_suppest}
\begin{alignat}{2}
& \left|l_1(\bw,\bv)\right|\leq C\sqrt{\nu}h\norm{\bw}_2\trinorm{\bv}_\mathcal{R},\label{eqn: ur_suppest1} \\
& \left|l_3(\bw,\bv)\right|\leq C\sqrt{\nu}h\norm{\bw}_2\trinorm{\bv}_\mathcal{R},\label{eqn: ur_suppest2} \\
& \left|l_4(\bw,\bv)\right|\leq C h\norm{\bw}_2\trinorm{\bv}_\mathcal{R},\label{eqn: ur_suppest3}\\
& \left|\mathbf{s}(\Pi_h\bw,\bv)\right|\leq C h^2\norm{\bw}_2\trinorm{\bv}_\mathcal{R},\label{eqn: ur_suppest4}
\end{alignat}
\end{subequations}
where $C$ is a generic positive constant independent of $\nu$ and $h$ and may vary in each case.
\end{lemma}

\begin{proof}
The estimates \eqref{eqn: ur_suppest1} and \eqref{eqn: ur_suppest4} are proved by the estimate in Lemma~\ref{lemma: eg_supp_est} and the norm equivalence \eqref{eqn: norm_eq_tri}.
On the other hand, the Cauchy-Schwarz inequality, \eqref{eqn: cR-err}, and \eqref{eqn: ur_norm_equiv} lead to
\begin{align*}
    \left|l_3(\bw,\bv)\right|&=\left|\nu(\Delta\bw, \cR\bv-\bv)_{\Th}\right|\\
    &\leq \nu\norm{\bw}_2\norm{\cR\bv-\bv}_0\\
    &\leq C\nu h\norm{\bw}_2\enorm{\bv}\\
    &\leq C\sqrt{\nu}h\norm{\bw}_2\trinorm{\bv}_\mathcal{R}.
\end{align*}
Using the Cauchy-Schwarz inequality, \eqref{eqn: cR-err}, \eqref{eqn: Pih_stability}, and \eqref{eqn: Pih_err},
we get the following upper bounds,
\begin{align*}
    \left|l_4(\bw,\bv)\right|&=\left|(\cR\Pi_h\bw-\bw,\cR\bv)_{\Th}\right|\\
    &\leq \left|(\cR\Pi_h\bw-\Pi_h\bw,\cR\bv)_{\Th}\right|+\left|(\Pi_h\bw-\bw,\cR\bv)_{\Th}\right|\\
    &\leq Ch\enorm{\Pi_h\bw}\norm{\cR\bv}_0+\norm{\Pi_h\bw-\bw}_0\norm{\cR\bv}_0\\
    &\leq Ch|\bw|_1\trinorm{\bv}_\mathcal{R}.
\end{align*}
\end{proof}

Hence, we prove error estimates of the \texttt{PR-EG} method in Algorithm~\ref{alg:UREG}.

\begin{theorem}\label{thm: ur_error_est1}
Let $(\bu,p)\in [H_0^1(\Omega)\cap H^2(\Omega)]^d\times [L_0^2(\Omega)\cap H^1(\Omega)]$ be the solution to \eqref{eqn: governing1}-\eqref{eqn: governing3}, and $(\buh,p_h)\in \Vh\times Q_h$ be the discrete solution from the \texttt{PR-EG} method. Then, we have the following pressure-robust error estimates
 \begin{align*}
      &\trinorm{\Pi_h\bu-\bu_h}_\mathcal{R}\leq Ch(\sqrt{\nu}+1)\norm{\bu}_2,\\
      &\norm{\mathcal{P}_0p-p_h}_0\leq C h(\nu+\sqrt{\nu})\norm{\bu}_2 + Ch^2\norm{\bu}_2.
  \end{align*}
\end{theorem}

\begin{proof}
 We start with the error equation \eqref{eqn: ur_erreqn1},
 \begin{equation*}
     \bb(\bv,\epsilon_h)=\nu\ba(\be_h,\bv)+\tilde{\bc}(\be_h,\bv)-l_1(\bu,\bv)-l_3(\bu,\bv)-l_4(\bu,\bv)-\mathbf{s}(\Pi_h\bu,\bv).
 \end{equation*}
Then, it follows from \eqref{eqn: ur_ac_conti} and \eqref{sys: ur_suppest} that 
\begin{equation*}
    \bb(\bv,\epsilon_h)\leq C\left(\trinorm{\be_h} _\mathcal{R}+\sqrt{\nu}h\norm{\bu}_2+h\norm{\bu}_2+h^2\norm{\bu}_2\right)\trinorm{\bv}_\mathcal{R}.
\end{equation*}
 From the inf-sup condition \eqref{eqn: ur_inf-sup} with \eqref{eqn: bound_cnb}, we obtain
 \begin{equation}
     \norm{\epsilon_h}_0\leq C(\sqrt{\nu}+h)\left(\trinorm{\be_h}_\mathcal{R}+\sqrt{\nu}h\norm{\bu}_2+h\norm{\bu}_2+h^2\norm{\bu}_2\right).\label{eqn: inter_result1}
 \end{equation}
 We also choose $\bv=\be_h$ and $q=\epsilon_h$ in \eqref{sys: ur_erreqn} and substitute \eqref{eqn: ur_erreqn2} into \eqref{eqn: ur_erreqn1} to get
 \begin{equation*}
     \nu\ba(\be_h,\be_h)+\tilde{\bc}(\be_h,\be_h)=-\bb(\bm{\chi}_h,\epsilon_h)+l_1(\bu,\be_h)+l_3(\bu,\be_h)+l_4(\bu,\be_h)+\mathbf{s}(\Pi_h\bu,\be_h).
 \end{equation*}
Here, it follows from \eqref{eqn: supp_conti2} that
\begin{equation}
    |\bb(\bm{\chi}_h,\epsilon_h)|\leq Ch\norm{\bu}_2\norm{\epsilon_h}_0.\label{eqn: inter_result2}
\end{equation}
Therefore, from \eqref{eqn: ur_ac_coer}, \eqref{sys: ur_suppest}, and \eqref{eqn: inter_result2}, we have
\begin{equation*}
    \trinorm{\be_h}_\mathcal{R}^2\leq C\left( h\norm{\bu}_2\norm{\epsilon_h}_0+\sqrt{\nu}h\norm{\bu}_2\trinorm{\be_h}_\mathcal{R}+h\norm{\bu}_2\trinorm{\be_h}_\mathcal{R}\right),
\end{equation*}
while omitting $h^2$-terms.
We also replace $\norm{\epsilon_h}_0$ by its upper bound in \eqref{eqn: inter_result1} omitting high-order terms,
\begin{equation*}
    \trinorm{\be_h}^2_\mathcal{R}\leq C\left(\sqrt{\nu}h\norm{\bu}_2\trinorm{\be_h}_\mathcal{R}+h\norm{\bu}_2\trinorm{\be_h}_\mathcal{R}\right).
\end{equation*}
In this case, the Young's inequality gives
\begin{equation*}
    \sqrt{\nu}h\norm{\bu}_2\trinorm{\be_h}_\mathcal{R}\leq\frac{\nu h^2}{2\alpha}\norm{\bu}_2^2+\frac{\alpha}{2}\trinorm{\be_h}^2_\mathcal{R},\quad
    h\norm{\bu}_2\trinorm{\be_h}_\mathcal{R}\leq\frac{h^2}{2\alpha}\norm{\bu}_2^2+\frac{\alpha}{2}\trinorm{\be_h}^2_\mathcal{R}.
\end{equation*}
Therefore, it follows from choosing a proper $\alpha$ that
\begin{equation*}
    \trinorm{\be_h}^2_\mathcal{R}\leq Ch^2(\nu+1)\norm{\bu}_2^2,
\end{equation*}
which implies that
\begin{equation*}
    \trinorm{\be_h}_\mathcal{R}\leq Ch(\sqrt{\nu}+1)\norm{\bu}_2.
\end{equation*}
If we apply this estimate to \eqref{eqn: inter_result1}, then we obtain
\begin{equation*}
    \norm{\epsilon_h}_0\leq Ch(\nu+\sqrt{\nu})\norm{\bu}_2+Ch^2\norm{\bu}_2.
\end{equation*}
\end{proof}

\begin{remark}
    We emphasize that the error bounds in Theorem~\ref{thm: ur_error_est1} are pressure-robust and have no detrimental effect from small $\nu$.
With $\nu\rightarrow0$, the \texttt{PR-EG} method's velocity errors decrease in the optimal order, and pressure errors do in the second order (superconvergence is expected).
This result implies that the \texttt{PR-EG} method produces stable and accurate solutions to the Brinkman equations in the Darcy regime.
\end{remark}

In addition, we prove total error estimates showing the optimal orders of convergence in velocity and pressure.
\begin{theorem}
Under the same assumption of Theorem~\ref{thm: ur_error_est1}, we have the following error estimates
\begin{align*}
    &\trinorm{\bu-\bu_h}_\mathcal{R}\leq Ch(\sqrt{\nu}+1)\norm{\bu}_2,\\
    &\norm{p-p_h}_0\leq Ch\left((\nu+\sqrt{\nu})\norm{\bu}_2+\norm{p}_1\right).
\end{align*}
\end{theorem}

\begin{proof}
For the velocity error estimate, we show
\begin{equation*}
    \trinorm{\bu-\Pi_h\bu}_\mathcal{R}\leq C\sqrt{\nu}h\norm{\bu}_2.
\end{equation*}
More precisely, we recall $\bm{\chi}_h=\bu-\Pi_h\bu$ and observe the energy norm,
\begin{equation*}
\trinorm{\bm{\chi}_h}^2_\mathcal{R}=\nu\enorm{\bm{\chi}_h}^2+\norm{\cR\bm{\chi}_h}_0^2+\rho_2\norm{h_e^{1/2}\jump{\bm{\chi}_h}}_{0,\Eh}^2.
\end{equation*}
Then, it follows from \eqref{eqn: cR-err}, \eqref{eqn: Pih_energy_err}, and \eqref{eqn: Pih_err} that
\begin{equation*}
    \norm{\cR\bm{\chi}_h}_0\leq \norm{\cR\bm{\chi}_h-\bm{\chi}_h}_0+\norm{\bm{\chi}_h}_0\leq Ch\enorm{\bm{\chi}_h}+\norm{\bm{\chi}_h}_0\leq Ch^2\norm{\bu}_2.
\end{equation*}
Also, from \eqref{eqn: trace} and \eqref{eqn: Pih_err}, we obtain
\begin{equation*}
    \norm{h_e^{1/2}\jump{\bm{\chi}_h}}_{0,\Eh}\leq C\left(\norm{\bm{\chi}_h}_{0}^2+h^2\norm{\nabla\bm{\chi}_h}_{0,\Th}^2\right)^{1/2}\leq Ch^2\norm{\bu}_2.
\end{equation*}
Hence, since $\enorm{\bm{\chi}_h}\leq Ch\norm{\bu}_2$, the error bound is
\begin{equation*}
    \trinorm{\bm{\chi}_h}_\mathcal{R}\leq C\left(\sqrt{\nu}h+h^2\right)\norm{\bu}_2.
\end{equation*}
Furthermore, the pressure error estimate is readily proved by the triangle inequality and interpolation error estimate \eqref{eqn: P_err}.
\end{proof}

In conclusion, the proposed \texttt{PR-EG} method solves the Brinkman equations in both the Stokes and Darcy regimes, having the optimal order of convergence for both velocity and pressure.


\section{Numerical Experiments}
\label{sec:num_exp}

This section shows numerical experiments validating our theoretical results with two- and three-dimensional examples.
The numerical methods in this paper and their discrete solutions are denoted as follows:
\begin{itemize}
    \item $(\bu_h^{\texttt{ST}},p_h^{\texttt{ST}})$: Solution by the \texttt{ST-EG} method in Algorithm~\ref{alg:EG}.
    \item $(\bu_h^{\texttt{PR}},p_h^{\texttt{PR}})$: Solution by the \texttt{PR-EG} method in Algorithm~\ref{alg:UREG}.
\end{itemize}
While considering the scaled Brinkman equations \eqref{sys: scaled_Brinkman} with the parameter $\nu$, we recall the error estimates for the \texttt{ST-EG} method in Theorem~\ref{thm: eg_error_est1},
\begin{subequations}\label{sys: sum_err_esti_st}
\begin{alignat}{2}
& \trinorm{\Pi_h\bu-\bu^{\texttt{ST}}_h}\lesssim(\sqrt{\nu}+1)h\norm{\bu}_2 + \left( h+\frac{h}{\sqrt{\nu+c_1 h^2}}\right)\norm{p}_1, \label{eqn: sum_err_esti_st_u}\\
& \norm{\Pz p-p_h^{\texttt{ST}}}_0\lesssim (\nu+\sqrt{\nu})h\norm{\bu}_2 + (\sqrt{\nu}+1)h\norm{p}_1,\label{eqn: sum_err_esti_st_p}
\end{alignat}
\end{subequations}
and the error estimates for the \texttt{PR-EG} method from Theorem~\ref{thm: ur_error_est1}
\begin{subequations}\label{sys: sum_err_esti_ur}
\begin{alignat}{2}
& \trinorm{\Pi_h\bu-\bu_h^{\texttt{PR}}}\lesssim (\sqrt{\nu}+1)h\norm{\bu}_2, \label{eqn: sum_err_esti_ur_u}\\
& \norm{\Pz p-p_h^{\texttt{PR}}}_0\lesssim (\nu+\sqrt{\nu})h\norm{\bu}_2+h^2\norm{\bu}_2.\label{eqn: sum_err_esti_ur_p}
\end{alignat}
\end{subequations}
We mainly check the error estimates \eqref{sys: sum_err_esti_st} and \eqref{sys: sum_err_esti_ur} by showing various numerical experiments with $\nu$ and $h$.
We also display the difference between the numerical solutions for \texttt{ST-EG} and \texttt{PR-EG} in the Darcy regime, which shows that the \texttt{PR-EG} method is needed to obtain stable and accurate velocity solutions.
Moreover, we present permeability tests considering the Brinkman equations \eqref{sys:governing} with viscosity $\mu$ and permeability $K$ and applying both EG methods.
The permeability tests enhance the motivation of using the \texttt{PR-EG} method for the case of extreme $\mu$ or $K$.

We implement the numerical experiments using the authors' MATLAB codes developed based on iFEM \cite{CHE09}.
The penalty parameters are $\rho_1=\rho_2=3$ for all the numerical experiments.

\subsection{Two dimensional tests}\label{subsec: 2d_tests}

Let the computational domain be $\Omega=(0,1)\times (0,1)$. The velocity field and pressure are chosen as
\begin{equation*}
    \bu
    = \left(\begin{array}{c}
    10x^2(x-1)^2y(y-1)(2y-1) \\
    -10x(x-1)(2x-1)y^2(y-1)^2
    \end{array}\right),
    \quad
    p = 10(2x-1)(2y-1).
\end{equation*}
Then, the body force $\bbf$ and the Dirichlet boundary condition are obtained from \eqref{sys: scaled_Brinkman} using the exact solutions.

\subsubsection{Robustness and accuracy test}

We compare the \texttt{ST-EG} and \texttt{PR-EG} methods to see robustness and check their accuracy based on the error estimates \eqref{sys: sum_err_esti_st} and \eqref{sys: sum_err_esti_ur}.
First, we interpret the \texttt{ST-EG} method's velocity error estimate \eqref{eqn: sum_err_esti_st_u} depending on the relation between coefficient $\nu$ and mesh size $h$.
The first-order convergence of the energy norm with $h$ is guaranteed when $\nu\gg h^2$, but it is hard to tell any order of convergence when $\nu$ is smaller than $h^2$ due to the term $h/\sqrt{\nu+c_1h^2}$.
On the other hand, the velocity error estimate for the \texttt{PR-EG} method \eqref{eqn: sum_err_esti_ur_u} means the first-order convergence in $h$ regardless of $\nu$.
\begin{figure}[!htb]
\centering
\begin{subfigure}{0.4\linewidth}
    \centering
\includegraphics[width=1\linewidth]{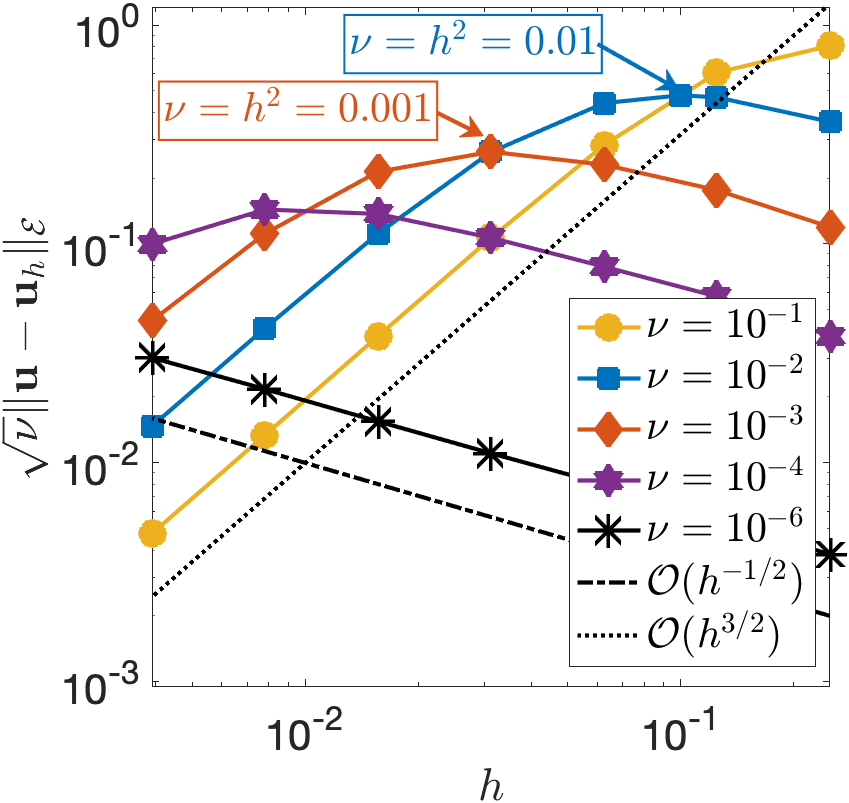}
    \caption{\texttt{ST-EG}: Velocity errors with $\nu$ and $h$}
    \end{subfigure}
    \hskip 20pt
    \begin{subfigure}{0.4\linewidth}
    \centering
\includegraphics[width=1\linewidth]{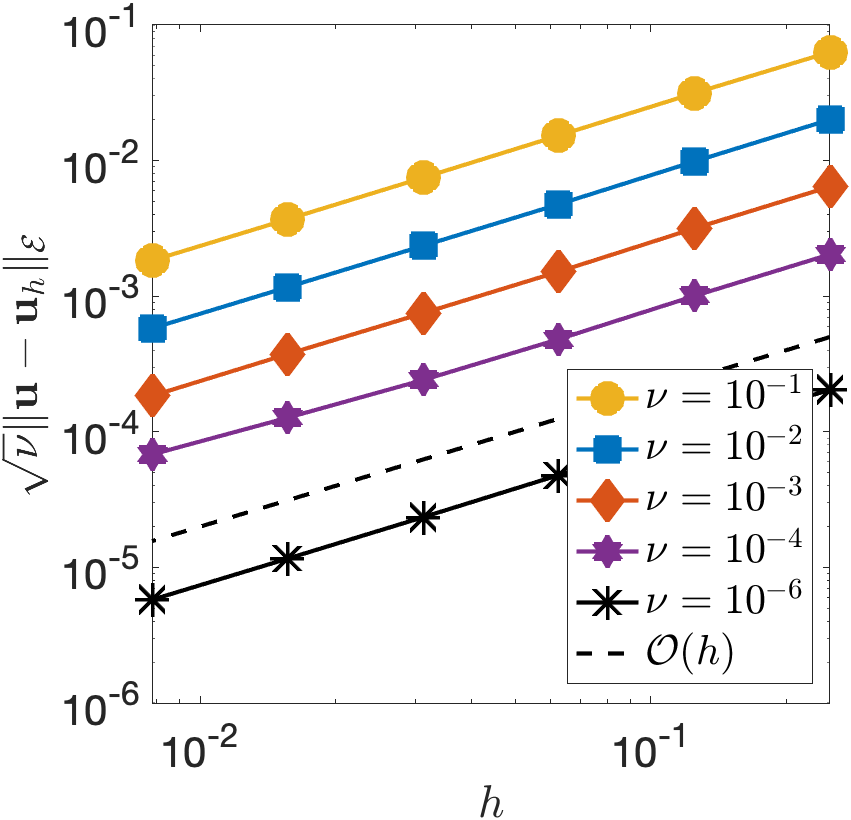}
        \caption{\texttt{PR-EG}: Velocity errors with $\nu$ and $h$}
    \end{subfigure}
    \caption{Convergence history with $h$ when different $\nu$ is given.}
    \label{figure: robust_test}
\end{figure}
In Figure~\ref{figure: robust_test}, we check the discrete $H^1$-error for the velocity scaled by $\nu$, $\sqrt{\nu}\enorm{\bu-\bu_h}$. It is a component of the energy norm $\trinorm{\bu-\bu_h}$.
The \texttt{ST-EG} method tends to produce errors increasing with $\mathcal{O}(h^{-1/2})$ when $h>\sqrt{\nu}$, while the errors decrease with $\mathcal{O}(h^{3/2})$ when $h<\sqrt{\nu}$.
This result supports the error estimates \eqref{eqn: sum_err_esti_st_u} (superconvergence may happen because we solve the problem on structured meshes) and means that a tiny mesh size is needed for accurate solutions with small $\nu$.
However, the \texttt{PR-EG} method's errors uniformly show the first-order convergence, $\mathcal{O}(h)$, regardless of $\nu$.
This result supports the error estimates \eqref{eqn: sum_err_esti_ur_u}, so the \texttt{PR-EG} method guarantees stable and accurate solutions in both the Stokes and Darcy regimes.

\begin{table}[ht]
    \centering
    \begin{tabular}{|c||c|c|c|c|c|c|c|c|}
    \hline
        &  \multicolumn{6}{c|}{\texttt{ST-EG}
        } \\
    \cline{2-7}   
      $h$  & {$\trinorm{\bu-\bu_h^{\texttt{ST}}}$} & {\small Order} & { $\sqrt{\nu}\|\mathbf{u}-\mathbf{u}_h^{\texttt{ST}}\|_\mathcal{E}$} & {\small Order} & {$\|\bu-\bu_h^{\texttt{ST}}\|_0$} & {\small Order} \\ 
      \hline
      $1/4$ & 9.695e-1 & -  & 4.437e-3 & - & 1.763e-1 & - \\
      \hline
      $1/8$  & 7.130e-1 & 0.44   & 6.645e-3 & -0.58 & 1.619e-1 & 0.12  \\
      \hline
      $1/16$  & 4.939e-1 & 0.53   & 9.015e-3 & -0.44 & 9.999e-2 & 0.70 \\
      \hline
      $1/32$  & 3.430e-1 & 0.53 & 1.234e-2 & -0.45 & 6.154e-2 & 0.70 \\
      \hline
      $1/64$  & 2.402e-1 & 0.51  &  1.715e-2 & -0.48 & 4.065e-2 & 0.60 \\
      \hline
       &  \multicolumn{6}{c|}{\texttt{PR-EG}
        } \\
    \cline{2-7}   
      $h$  & {$\trinorm{\bu-\bu_h^{\texttt{PR}}}$} & {\small Order} & { $\sqrt{\nu}\|\mathbf{u}-\mathbf{u}_h^{\texttt{PR}}\|_\mathcal{E}$} & {\small Order} & {$\|\bu-\bu_h^{\texttt{PR}}\|_0$} & {\small Order} \\ 
      \hline
      $1/4$ & 2.479e-2 & -  & 2.045e-4 & - & 1.844e-2 & - \\
      \hline
      $1/8$  & 4.774e-3 & 2.38   & 1.003e-4 & 1.03 & 2.727e-3 & 2.76  \\
      \hline
      $1/16$  & 8.126e-4 & 2.55   & 4.797e-5 & 1.06 & 5.257e-4 & 2.38 \\
      \hline
      $1/32$  & 1.565e-4 & 2.38 & 2.346e-5 & 1.03 & 1.180e-4 & 2.16 \\
      \hline
      $1/64$  & 3.464e-5 & 2.18  &  1.160e-5 & 1.02 & 2.792e-5 & 2.08 \\
      \hline
    \end{tabular}
    \caption{A mesh refinement study for the velocity errors of the \texttt{ST-EG} and \texttt{PR-EG} with $h$ when $\nu=10^{-6}$.}
    \label{table: refine_study_2d_u}
\end{table}
We fix $\nu=10^{-6}$ and compare the velocity errors and solutions of the \texttt{ST-EG} and \texttt{PR-EG} methods.
Table~\ref{table: refine_study_2d_u} displays the energy errors and their major components, the discrete $H^1$-errors scaled by $\nu$ and $L^2$-errors.
For the \texttt{ST-EG} method, the energy errors decrease in the half-order convergence because the $L^2$-errors are dominant and decrease in the same order.
However, the $H^1$-errors keep increasing unless $h<\sqrt{\nu}=10^{-3}$, so the $H^1$-errors will become dominant and deteriorate the order of convergence of the energy errors.
On the other hand, using the \texttt{PR-EG} method, we expect from \eqref{eqn: sum_err_esti_ur_u} that the energy errors and major components converge in at least the first order of $h$.
Indeed, Table~\ref{table: refine_study_2d_u} shows that the $H^1$-errors decrease in the first order with $h$, while the $L^2$-errors reduce in the second order.
Since the energy error involve both $H^1$- and $L^2$-errors, the energy errors decrease in the second order because of the dominant $L^2$-errors but eventually converge in the first order coming from the $H^1$-errors.
\begin{figure}[!htb]
\begin{subfigure}{\linewidth}
\centering
    \includegraphics[width=0.3\linewidth]{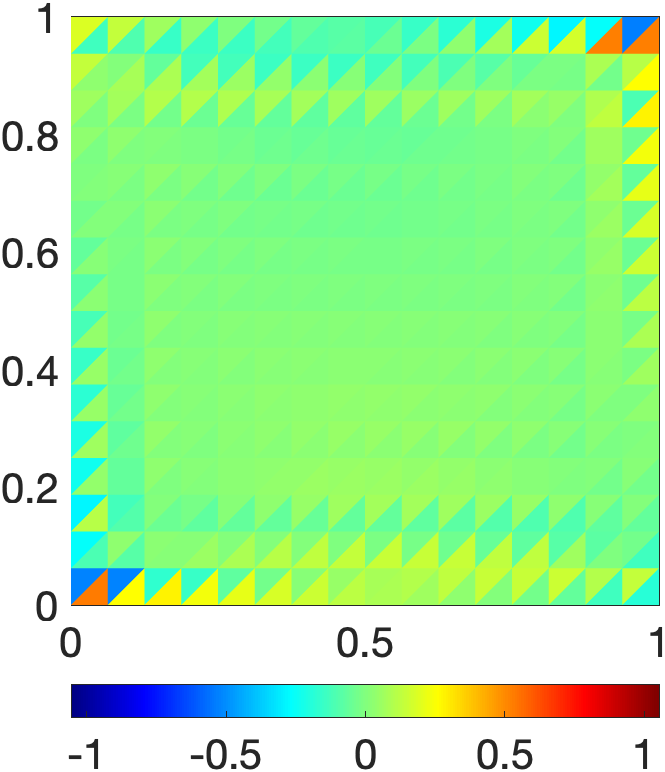}
    \hskip 15pt
    \includegraphics[width=0.3\linewidth]{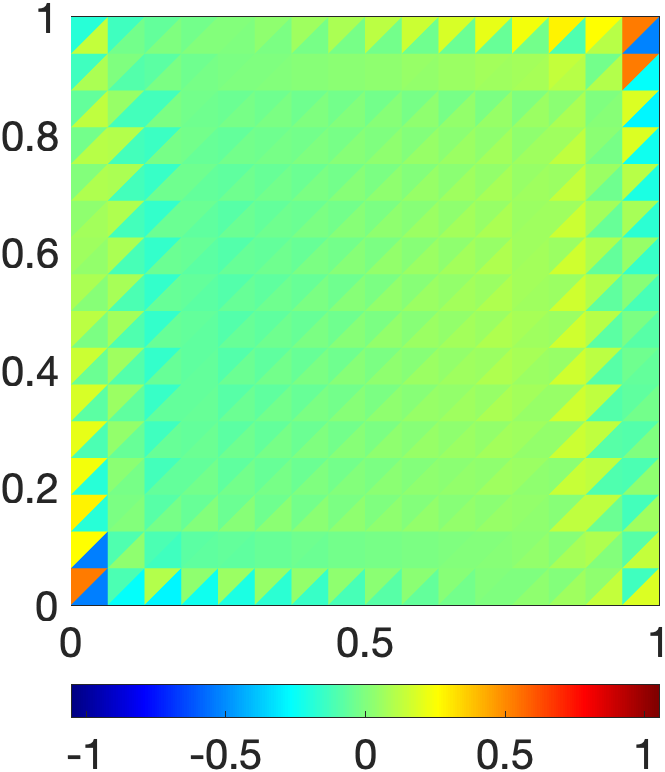}
    \hskip 15pt
    \includegraphics[width=0.3\linewidth]{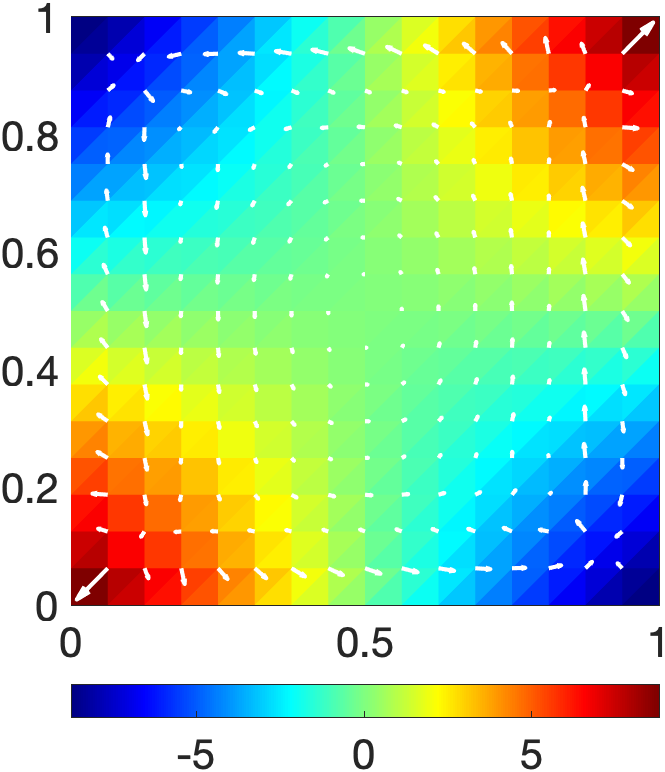}
    \caption{\texttt{ST-EG}: $u_1$, $u_2$, and $p$ with the velocity vector fields, from left to right}
\end{subfigure}
\vskip 20pt
\begin{subfigure}{\linewidth}
    \centering
    \includegraphics[width=0.3\linewidth]{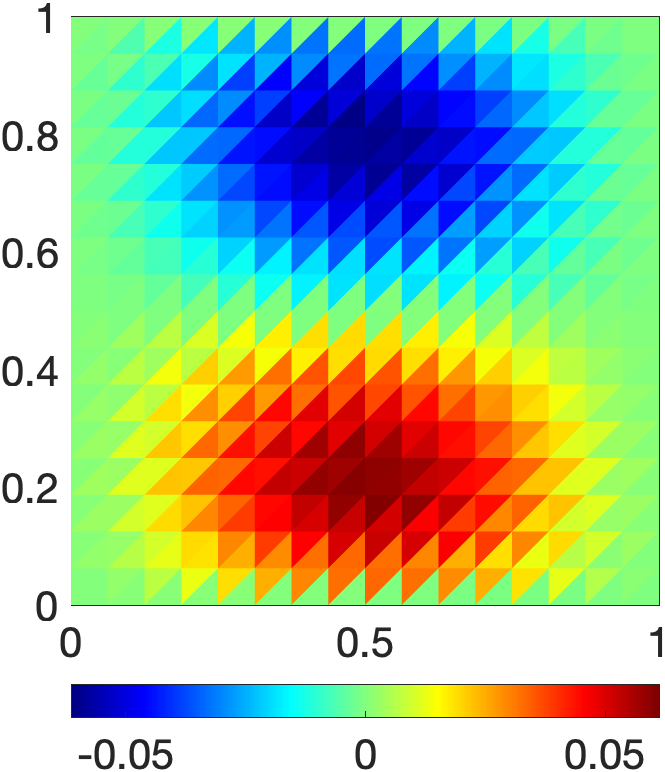}
    \hskip 15pt
    \includegraphics[width=0.3\linewidth]{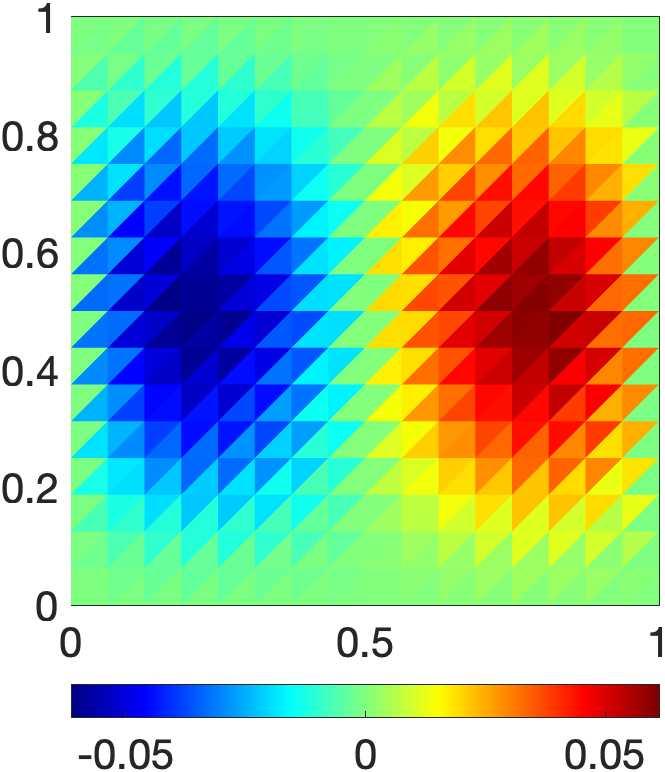}
    \hskip 15pt
    \includegraphics[width=0.3\linewidth]{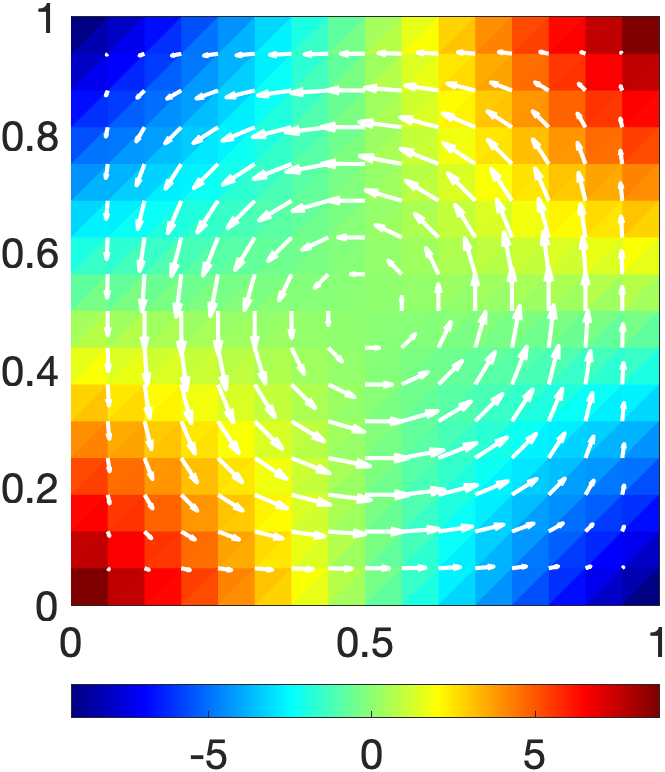}
    \caption{\texttt{PR-EG}: $u_1$, $u_2$, and $p$ with the velocity vector fields, from left to right}
\end{subfigure}
    \caption{Numerical solutions of \texttt{ST-EG} and \texttt{PR-EG} when $\nu=10^{-6}$ and $h=1/16$.}
    \label{figure: nume_sol_2d}
\end{figure}
In Figure~\ref{figure: nume_sol_2d}, the \texttt{PR-EG} method produces accurate velocity solutions clearly showing a vortex flow pattern when $\nu=10^{-6}$ and $h=1/16$. In contrast, the numerical velocity from the \texttt{ST-EG} method includes significant oscillations around the boundary of the domain.

\begin{table}[!htb]
    \centering
    \begin{tabular}{|c||c|c|c|c|c|c|c|c|}
    \hline
        &  \multicolumn{4}{c|}{\texttt{ST-EG}
        } &  \multicolumn{4}{c|}{\texttt{PR-EG}
        } \\
    \cline{2-9}   
      $h$  & {\small $\|\Pz p-p_h^{\texttt{ST}}\|_0$} & {\small Order} & {$\|p-p_h^{\texttt{ST}}\|_0$} & {\small Order} & {\small $\|\Pz p-p_h^{\texttt{PR}}\|_0$} & {\small Order} & {$\|p-p_h^{\texttt{PR}}\|_0$} & {\small Order} \\ 
      \hline
      $1/4$ & 5.783e-1 & -  & 1.116e+0 & - & 1.110e-2 & - & 9.548e-1 & -\\
      \hline
      $1/8$  & 1.682e-1 & 1.78   & 5.088e-1 & 1.13 & 7.762e-4 & 3.84 & 4.802e-1 & 0.99\\
      \hline
      $1/16$  & 5.455e-2 & 1.62   & 2.466e-1 & 1.04 & 3.756e-5 & 4.37 & 2.404e-1 & 1.00\\
      \hline
      $1/32$  & 1.917e-2 & 1.51 & 1.218e-1 & 1.02 & 2.408e-6 & 3.96 & 1.203e-1 & 1.00\\
      \hline
      $1/64$  & 7.271e-3 & 1.40  &  6.058e-2 & 1.01 & 2.089e-7 & 3.53 & 6.014e-2 & 1.00\\
      \hline
    \end{tabular}
    \caption{A mesh refinement study for the pressure errors of the \texttt{ST-EG} and \texttt{PR-EG} with $h$ when $\nu=10^{-6}$.}
    \label{table: refine_study_2d_p}
\end{table}
Moreover, the pressure error estimates \eqref{eqn: sum_err_esti_st_p} and \eqref{eqn: sum_err_esti_ur_p} tell us that the convergence order for the pressure errors is at least $\mathcal{O}(h)$ in both methods. However, the \texttt{PR-EG} method can produce superconvergent pressure errors because the term $h^2\norm{p}_1$ is dominant when $\nu$ is small.
In Table~\ref{table: refine_study_2d_p}, the pressure errors of the \texttt{PR-EG} method, $\norm{\Pz p-p_h^{\texttt{PR}}}_0$, decrease in at least $\mathcal{O}(h^3)$, which means superconvergence compared to the interpolation error estimate \eqref{eqn: P_err}.
On the other hand, the \texttt{ST-EG} method still yields pressure errors converging in the first order with $h$.
Since the interpolation error is dominant in the total pressure errors $\norm{p-p_h}_0$, the errors in Table~\ref{table: refine_study_2d_p} have the first-order convergence with $h$ in both methods.
Therefore, the numerical results support the pressure error estimates \eqref{eqn: sum_err_esti_st_p} and \eqref{eqn: sum_err_esti_ur_p}.

\subsubsection{Error profiles with respect to $\nu$}

We shall confirm the error estimates \eqref{sys: sum_err_esti_st} and \eqref{sys: sum_err_esti_ur} in terms of the parameter $\nu$ by checking error profiles depending on $\nu$.
We define the following error profile functions of $\nu$ based on the error estimates and show that these functions explain the behavior of the velocity and pressure errors with $\nu$:
\begin{itemize}
    \item $\displaystyle E_{\bu,2}^\texttt{ST}(\nu):=0.1h\sqrt{\nu}+\frac{0.3h}{\sqrt{\nu+3h^2}}+0.4h=\frac{0.1}{32}\sqrt{\nu}+\frac{0.3}{\sqrt{32^2\nu+3}}+\frac{0.4}{32}$ from \eqref{eqn: sum_err_esti_st_u},
    \item $\displaystyle E_{\bu,2}^\texttt{PR}(\nu):=0.8h\sqrt{\nu}+0.05h=\frac{0.8}{32}\sqrt{\nu}+\frac{0.05}{32}$ from \eqref{eqn: sum_err_esti_ur_u},
    \item $\displaystyle E_{p,2}^\texttt{ST}(\nu):=2h\nu+3h\sqrt{\nu}+0.3h=\frac{2}{32}\nu+\frac{3}{32}\sqrt{\nu}+\frac{0.3}{32}$ from \eqref{eqn: sum_err_esti_st_p},
    \item $\displaystyle E_{p,2}^\texttt{PR}(\nu):=0.5h\nu+0.01h\sqrt{\nu}+0.01h^2=\frac{0.5}{32}\nu+\frac{0.01}{32}\sqrt{\nu}+\frac{0.01}{32^2}$ from \eqref{eqn: sum_err_esti_ur_p},
\end{itemize}
where $h=1/32$.
\begin{figure}[!htb]
\centering
\begin{subfigure}{0.4\linewidth}
    \centering
    \includegraphics[width=1\textwidth]{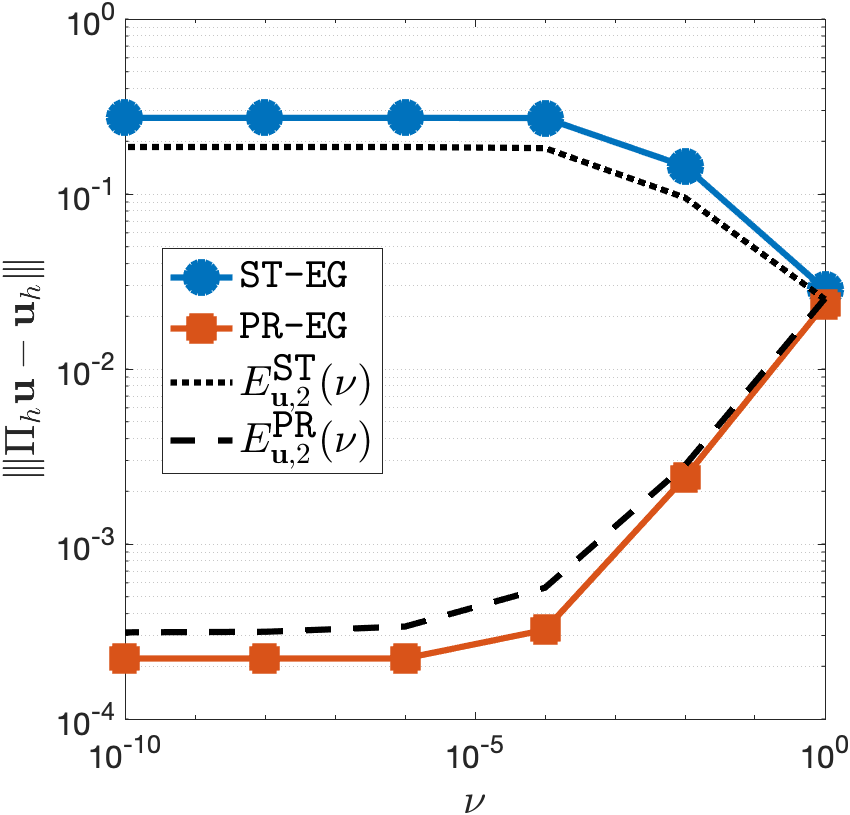}
    \caption{Velocity errors vs. $\nu$}
    \end{subfigure}
    \hskip 20pt
    \begin{subfigure}{0.4\linewidth}
    \centering
    \includegraphics[width=1\textwidth]{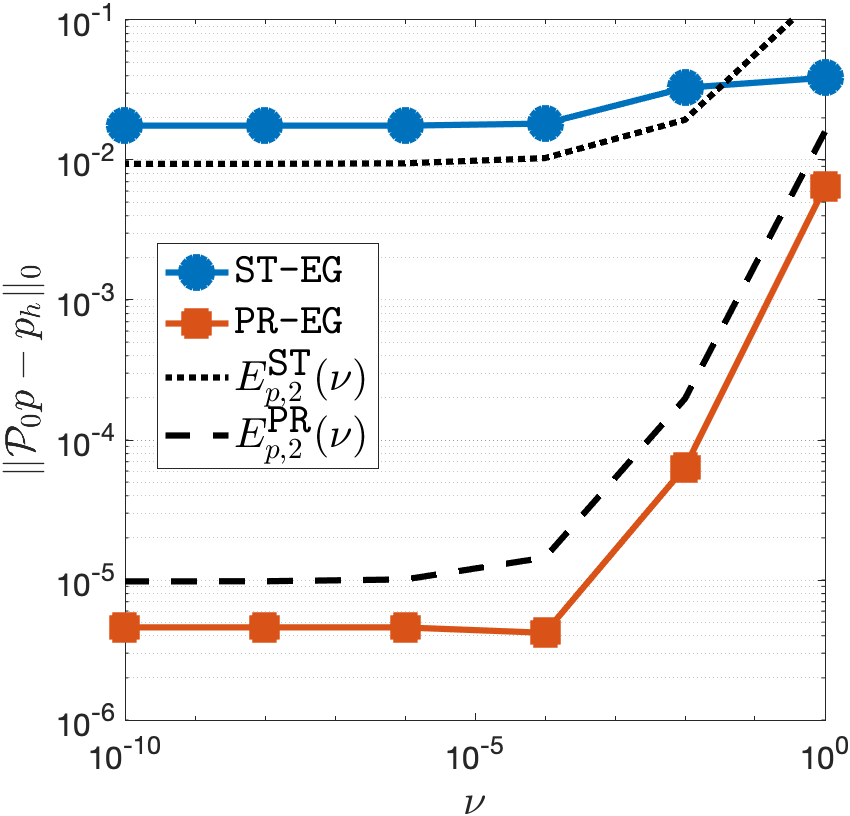}
        \caption{Pressure errors vs. $\nu$}
    \end{subfigure}
    \caption{Error profiles of the \texttt{ST-EG} and \texttt{PR-EG} methods with varying $\nu$ and a fixed mesh size $h=1/32$.}
    \label{figure: errors_prof_2d}
\end{figure}
Figure~\ref{figure: errors_prof_2d} shows the velocity and pressure errors and the graphs of the above error profile functions when $\nu$ decreases from $1$ to $0$ and $h=1/32$.
As shown in Figure~\ref{figure: errors_prof_2d}, the velocity errors for the \texttt{ST-EG} method increase when $\nu$ is between 1 to $10^{-4}$ and tend to remain constant when $\nu$ is smaller.
The \texttt{ST-EG} method's pressure errors decrease slightly and stay the same as $\nu\rightarrow0$.
On the other hand, the velocity and pressure errors for the \texttt{PR-EG} method significantly reduce and remain the same after $\nu=10^{-4}$.
This error behavior can be explained by the graphs of the error profile functions guided by the error estimates \eqref{sys: sum_err_esti_st} and \eqref{sys: sum_err_esti_ur}, so this result supports the estimates concerning $\nu$.
In addition, the velocity and pressure errors for the \texttt{PR-EG} method are almost 1000 times smaller than the \texttt{ST-EG} method in Figure~\ref{figure: errors_prof_2d}.
Therefore, we confirm that the \texttt{PR-EG} method guarantees more accurate solutions for velocity and pressure when $\nu$ is small.

\subsubsection{Permeability test}
In this test, we consider the Brinkman equations \eqref{sys:governing} with viscosity $\mu=10^{-6}$ and permeability given as the permeability map in Figure~\ref{figure: permeability_t1}.
\begin{figure}[!htb]
\centering
\includegraphics[width=0.3\linewidth]{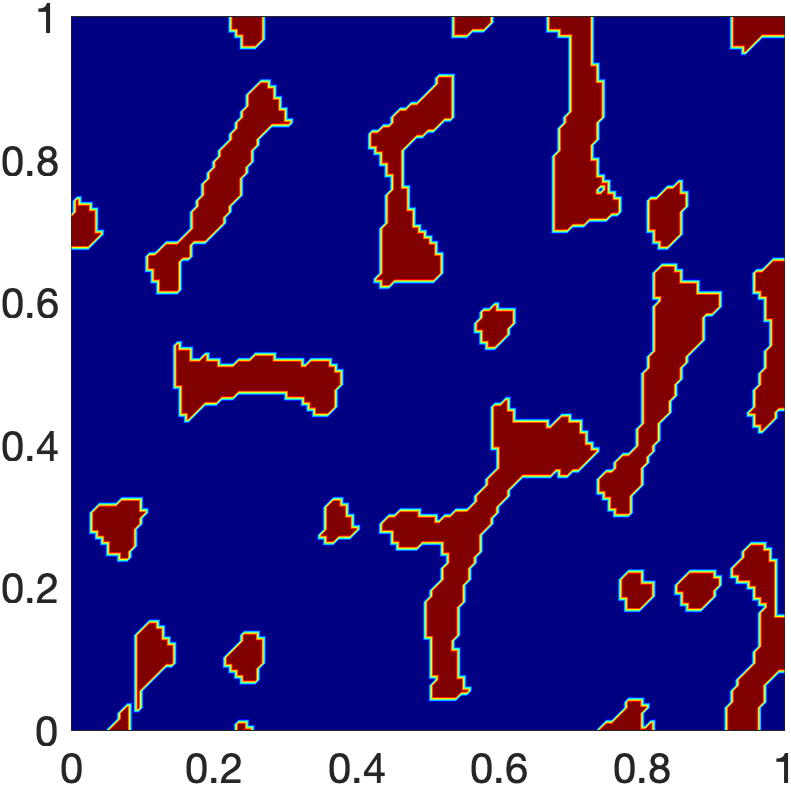}
    \caption{Permeability map; red regions mean $K^{-1}=10^6$ and blue regions mean $K^{-1}=1$.}
    \label{figure: permeability_t1}
\end{figure}
The permeability map indicates that fluid tends to flow following the blue regions, so the magnitude of numerical velocity will be more significant in the blue areas than in the red parts.
We set the velocity on the boundary of the domain as $\bu=\langle 1,0\rangle$ and body force as $\bbf = \langle 1, 1\rangle$.
\begin{figure}[!htb]
\centering
\begin{subfigure}{0.45\linewidth}
    \centering
    \includegraphics[width=1\linewidth]{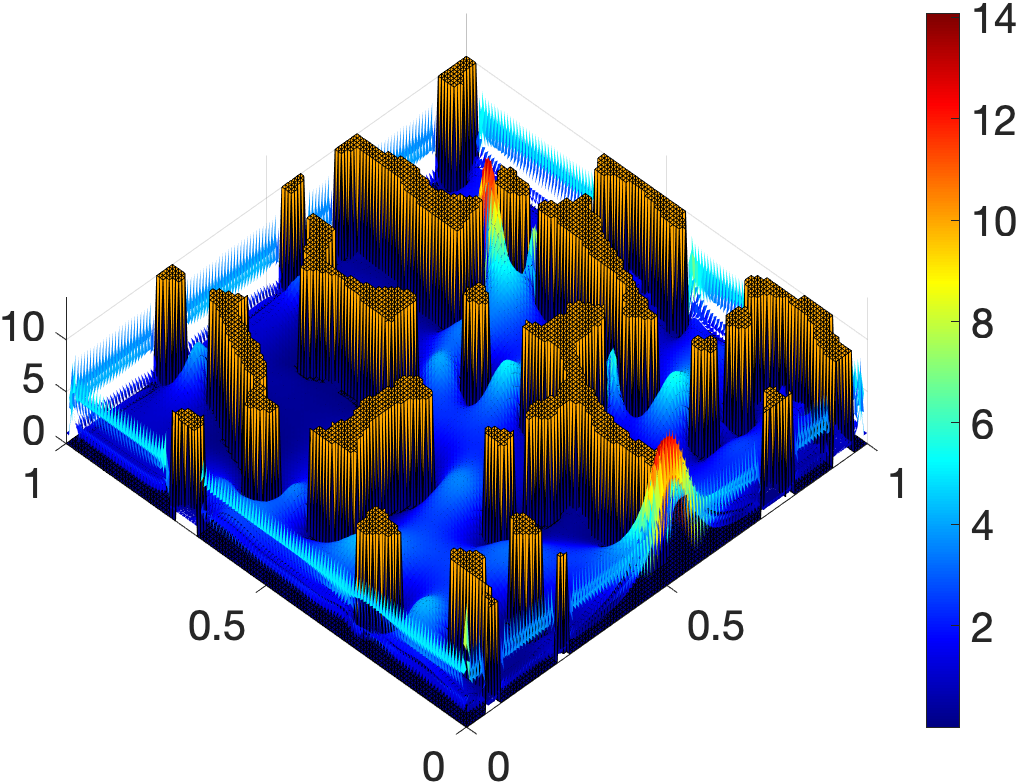}
    \caption{\texttt{ST-EG}: Magnitude of numerical velocity}
    \end{subfigure}
    \begin{subfigure}{0.45\linewidth}
    \centering
    \includegraphics[width=1\linewidth]{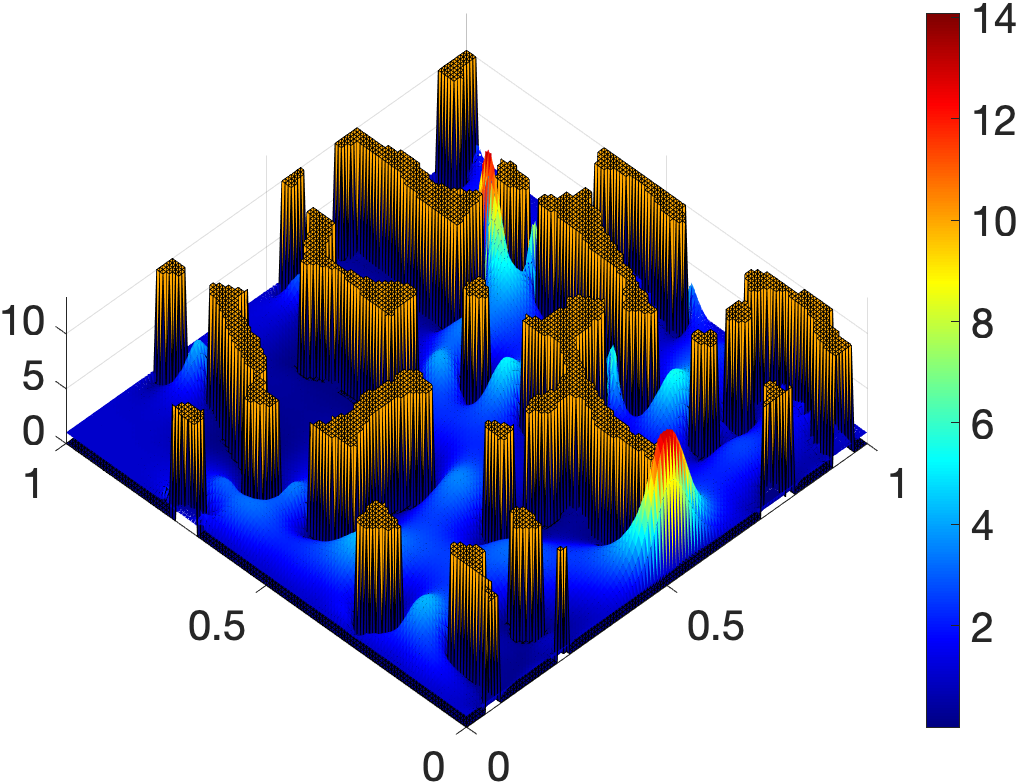}
        \caption{\texttt{PR-EG}: Magnitude of numerical velocity}
    \end{subfigure}
    \caption{Numerical velocity solutions of \texttt{ST-EG} and \texttt{PR-EG} on the permeability map.}
    \label{figure: mag_vel_2d}
\end{figure}
We mainly compare the magnitude of the numerical velocity obtained from the two methods in Figure~\ref{figure: mag_vel_2d}.
We clearly see that the \texttt{PR-EG} method's velocity is more stable than the \texttt{ST-EG} method's velocity containing nonnegligible noises (or oscillations) around the boundary.
This result tells that the \texttt{PR-EG} method is necessary for stable and accurate velocity solutions to the Brinkman equations with extreme viscosity and permeability.

\subsection{Three dimensional tests}

    We consider a three-dimensional flow in a unit cube $\Omega=(0,1)^3$. The velocity field and pressure are chosen as
\begin{equation*}
    \bu 
    = \left(\begin{array}{c}
    \sin(\pi x)\cos(\pi y) - \sin(\pi x)\cos(\pi z) \\
    \sin(\pi y)\cos(\pi z) - \sin(\pi y)\cos(\pi x) \\
    \sin(\pi z)\cos(\pi x) - \sin(\pi z)\cos(\pi y)
    \end{array}\right),\quad
    p = \pi^3\sin(\pi x)\sin(\pi y)\sin(\pi z)-1.
\end{equation*}
The body force $\bbf$ and the Dirichlet boundary condition are given in the same manner as the two-dimensional example.

\subsubsection{Robustness and accuracy test}

In the two-dimensional tests, we checked that the condition $h<\sqrt{\nu}$ was required to guarantee the optimal order of convergence for the \texttt{ST-EG} method, while the \texttt{PR-EG} method showed a uniform performance in convergence independent of $\nu$.
We obtained the same result as in Figure~\ref{figure: robust_test} from this three-dimensional test.
\begin{table}[ht]
    \centering
    \begin{tabular}{|c||c|c|c|c|c|c|c|c|}
    \hline
        &  \multicolumn{6}{c|}{\texttt{ST-EG}
        } \\
    \cline{2-7}   
      $h$  & {$\trinorm{\bu-\bu_h^{\texttt{ST}}}$} & {\small Order} & { $\sqrt{\nu}\|\mathbf{u}-\mathbf{u}_h^{\texttt{ST}}\|_\mathcal{E}$} & {\small Order} & {$\|\bu-\bu_h^{\texttt{ST}}\|_0$} & {\small Order} \\ 
      \hline
      $1/4$ & 2.105e+0 & -  & 1.379e-2 & - & 4.534e-1 & - \\
      \hline
      $1/8$  & 1.627e+0 & 0.37 & 2.112e-2 & -0.62 & 3.829e-1 & 0.24  \\
      \hline
      $1/16$  & 1.172e+0 & 0.47 & 3.018e-2 & -0.52 & 2.800e-1 & 0.45 \\
      \hline
      $1/32$  & 8.219e-1 & 0.51 & 4.214e-2 & -0.48 & 1.852e-1 & 0.60 \\
      \hline
       &  \multicolumn{6}{c|}{\texttt{PR-EG}
        } \\
    \cline{2-7}   
      $h$  & {$\trinorm{\bu-\bu_h^{\texttt{PR}}}$} & {\small Order} & { $\sqrt{\nu}\|\mathbf{u}-\mathbf{u}_h^{\texttt{PR}}\|_\mathcal{E}$} & {\small Order} & {$\|\bu-\bu_h^{\texttt{PR}}\|_0$} & {\small Order} \\ 
      \hline
      $1/4$ & 3.738e-1 & -  & 2.684e-3 & - & 1.828e-1 & - \\
      \hline
      $1/8$  & 8.797e-2 & 2.09 & 1.346e-3 & 1.00 & 3.026e-2 & 2.59 \\
      \hline
      $1/16$  & 2.079e-2 &  2.08 & 6.600e-4 & 1.03 & 6.203e-3 & 2.29 \\
      \hline
      $1/32$  & 5.101e-3 & 2.03 & 3.256e-4 & 1.02 & 1.441e-3 & 2.11 \\
      \hline
    \end{tabular}
    \caption{A mesh refinement study for the velocity errors of the \texttt{ST-EG} and \texttt{PR-EG} with $h$ when $\nu=10^{-6}$.}
    \label{table: refine_study_3d_u}
\end{table}
Table~\ref{table: refine_study_3d_u} displays the velocity solutions' energy errors and influential components, comparing the \texttt{PR-EG} method with \texttt{ST-EG} when $\nu=10^{-6}$.
The \texttt{ST-EG} method's energy errors tend to decrease because the dominant $L^2$-errors decrease, but the $H^1$-errors scaled by $\nu$ increase.
These $H^1$-errors may make the energy errors nondecreasing until $h<\sqrt{\nu}=10^{-3}$. 
However, the \texttt{PR-EG} methods guarantee at least first-order convergence for all the velocity errors, showing much smaller errors than the \texttt{ST-EG} method.
This numerical result supports the velocity error estimates in \eqref{eqn: sum_err_esti_st_u} and \eqref{eqn: sum_err_esti_ur_u}, and we expect more accurate solutions from the \texttt{PR-EG} method when $\nu$ is small.
\begin{figure}[!htb]
\centering
\begin{subfigure}{0.4\linewidth}
    \centering
    \includegraphics[width=1\linewidth]{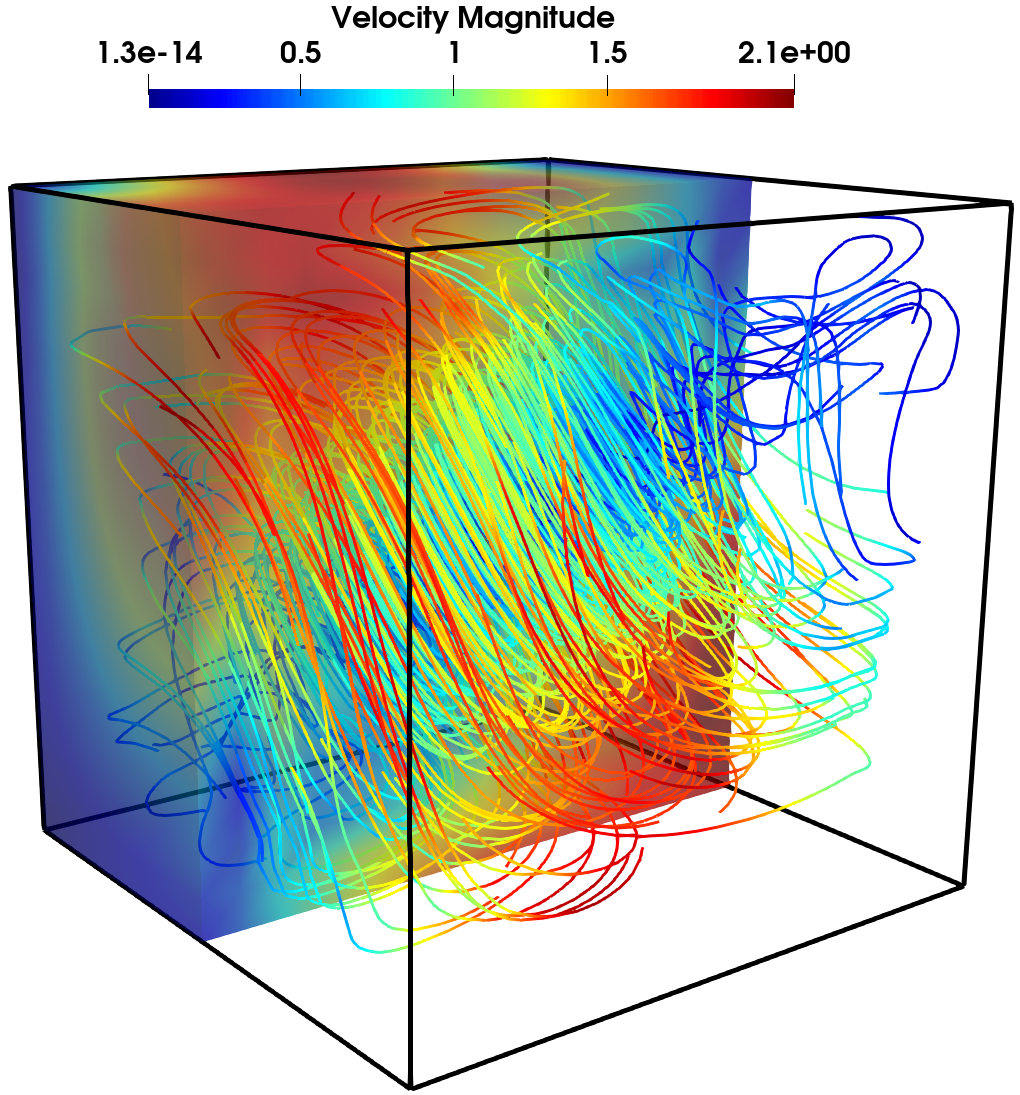}
    \caption{\texttt{ST-EG}: Streamlines and magnitude}
    \end{subfigure}
    \hskip 20pt
    \begin{subfigure}{0.4\linewidth}
    \centering
    \includegraphics[width=1\linewidth]{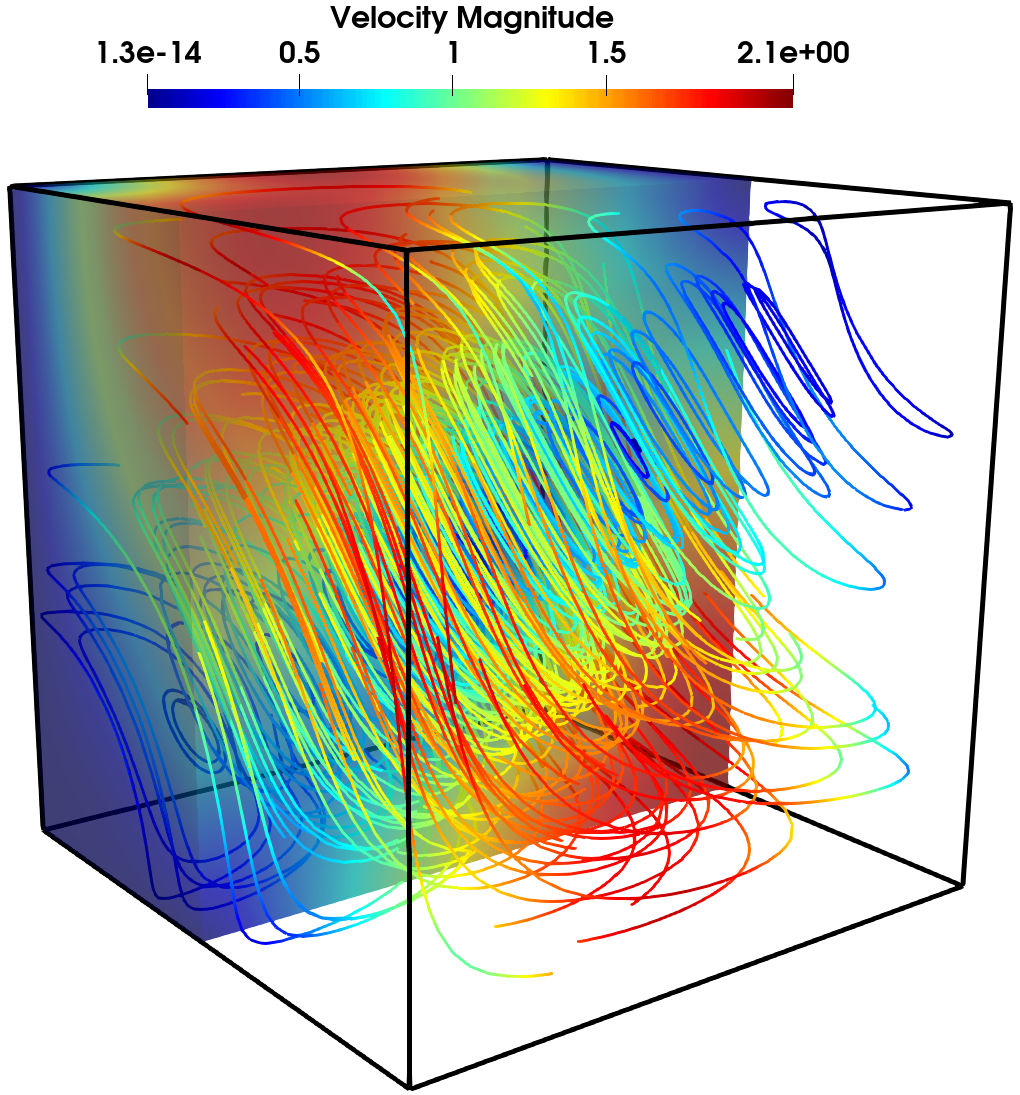}
        \caption{\texttt{PR-EG}: Streamlines and magnitude}
    \end{subfigure}
    \caption{Numerical velocity solutions of \texttt{ST-EG} and \texttt{PR-EG} when $\nu=10^{-6}$ and $h=1/16$.}
    \label{figure: nume_sol_3d}
\end{figure}
In addition, we compare numerical velocity solutions of the \texttt{ST-EG} and \texttt{PR-EG} methods when $\nu=10^{-6}$ and $h=1/16$ in Figure~\ref{figure: nume_sol_3d}.
The velocity solutions of both methods seem to capture a three-dimensional vortex flow expected from the exact velocity.
However, the velocity of the \texttt{ST-EG} method contains noises around the right-top and left-bottom corners, where the streamlines do not form a circular motion.

\begin{table}[ht]
    \centering
    \begin{tabular}{|c||c|c|c|c|c|c|c|c|}
    \hline
        &  \multicolumn{4}{c|}{\texttt{ST-EG}
        } &  \multicolumn{4}{c|}{\texttt{PR-EG}
        } \\
    \cline{2-9}   
      $h$  & {\small $\|\Pz p-p_h^{\texttt{ST}}\|_0$} & {\small Order} & {$\|p-p_h^{\texttt{ST}}\|_0$} & {\small Order} & {\small $\|\Pz p-p_h^{\texttt{PR}}\|_0$} & {\small Order} & {$\|p-p_h^{\texttt{PR}}\|_0$} & {\small Order} \\ 
      \hline
      $1/4$ & 1.346e+0 & -  & 3.262e+0 & - & 1.109e-1 & - & 2.973e+0 & -\\
      \hline
      $1/8$  & 4.983e-1 & 1.43  & 1.593e+0 & 1.03 & 1.241e-2 & 3.16 & 1.513e+0 & 0.98  \\
      \hline
      $1/16$  & 1.805e-1 & 1.47 & 7.810e-1 & 1.03 & 1.344e-3 & 3.21 & 7.598e-1 & 0.99 \\
      \hline
      $1/32$  & 6.216e-2 & 1.54 & 3.854e-1 & 1.02 & 1.609e-4 & 3.06 & 3.804e-1 & 1.00 \\
      \hline
    \end{tabular}
    \caption{A mesh refinement study for the pressure errors of the \texttt{ST-EG} and \texttt{PR-EG} with $h$ when $\nu=10^{-6}$.}
    \label{table: refine_study_3d_p}
\end{table}
In Table~\ref{table: refine_study_3d_p}, 
as expected in \eqref{eqn: sum_err_esti_st_p}, the \texttt{ST-EG} method's pressure errors decrease in at least first-order.
On the other hand, the \texttt{PR-EG} method's pressure errors, $\norm{\Pz p -p_h^{\mathtt{UR}}}_0$, decrease much faster, showing superconvergence.
This phenomenon is expected by the pressure estimate \eqref{eqn: sum_err_esti_ur_p} when $\nu$ is small.
Moreover, the orders of convergence of the total pressure errors, $\norm{p-p_h}_0$, 
for both methods are approximately one due to the interpolation error.

\subsubsection{Error profiles with respect to $\nu$}

We define error profile functions suitable for the three-dimensional test by determining constants in the estimates \eqref{sys: sum_err_esti_st} and \eqref{sys: sum_err_esti_ur}:

\begin{itemize}
    \item $\displaystyle E_{\bu,3}^\texttt{ST}(\nu):=0.1h\sqrt{\nu}+\frac{h}{\sqrt{\nu+3h^2}}+9h=\frac{0.1}{16}\sqrt{\nu}+\frac{1}{\sqrt{16^2\nu+3}}+\frac{9}{16}$ from \eqref{eqn: sum_err_esti_st_u}
    \item $\displaystyle E_{\bu,3}^\texttt{PR}(\nu):=6h\sqrt{\nu}+0.25h=\frac{6}{16}\sqrt{\nu}+\frac{0.25}{16}$ from \eqref{eqn: sum_err_esti_ur_u},
    \item $\displaystyle E_{p,3}^\texttt{ST}(\nu):=1.5h\nu+h\sqrt{\nu}+2.5h=\frac{1.5}{16}\nu+\frac{1}{16}\sqrt{\nu}+\frac{2.5}{16}$ from \eqref{eqn: sum_err_esti_st_p},
    \item $\displaystyle E_{p,3}^\texttt{PR}(\nu):=2h\nu+0.02h\sqrt{\nu}+0.2h^2 = \frac{2}{16}\nu+\frac{0.02}{16}\sqrt{\nu}+\frac{0.2}{16^2}$ from \eqref{eqn: sum_err_esti_ur_p},
\end{itemize}
where $h=1/16$.
\begin{figure}[!htb]
\centering
\begin{subfigure}{0.4\linewidth}
    \centering
    \includegraphics[width=1\textwidth]{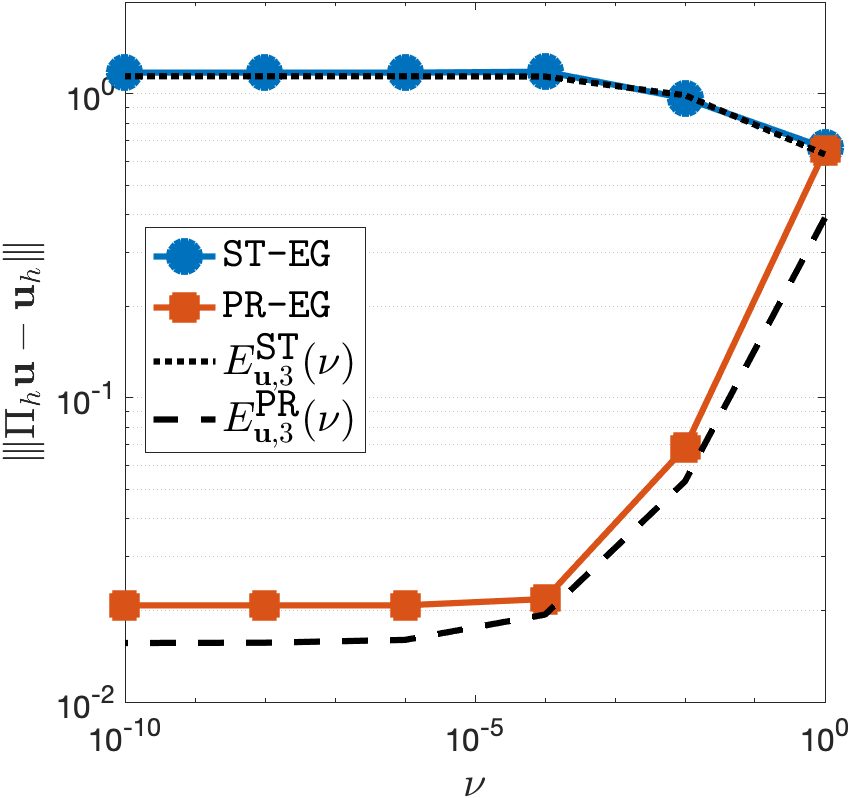}
    \caption{Velocity errors vs. $\nu$}
    \end{subfigure}
    \hskip 20pt
    \begin{subfigure}{0.4\linewidth}
    \centering
    \includegraphics[width=1\textwidth]{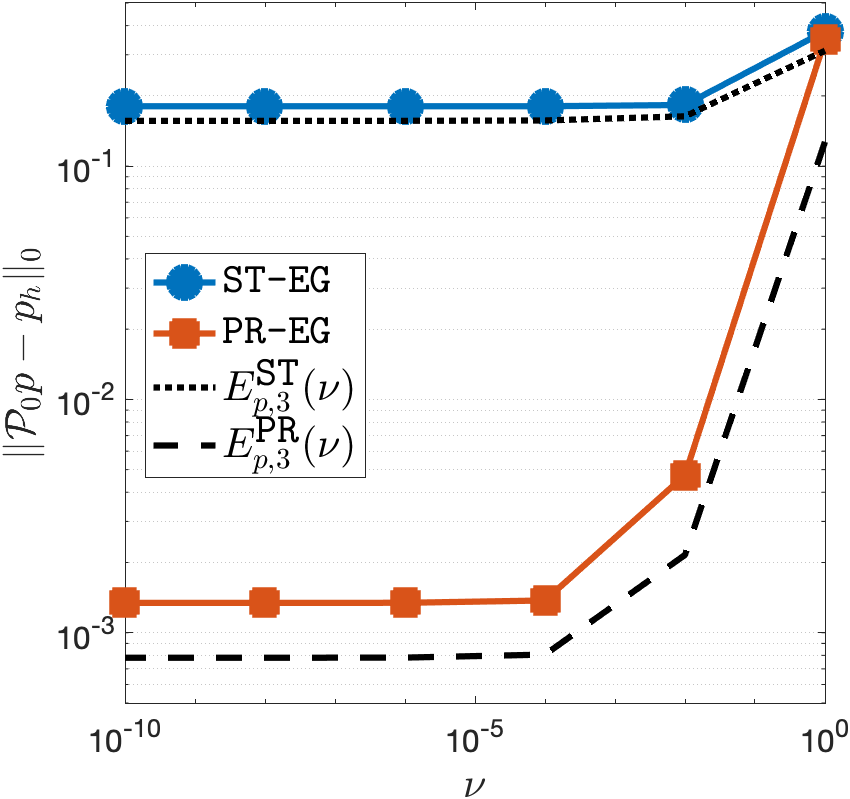}
        \caption{Pressure errors vs. $\nu$}
    \end{subfigure}
    \caption{Error profiles of the \texttt{ST-EG} and \texttt{PR-EG} methods with varying $\nu$ and a fixed mesh size $h=1/16$.}
    \label{figure: errors_prof_3d}
\end{figure}
In Figure~\ref{figure: errors_prof_3d}, the \texttt{PR-EG} method's velocity and pressure errors decrease when $\nu$ changes from 1 to $10^{-4}$ and remain the same when $\nu$ gets smaller.
However, the errors for the \texttt{ST-EG} method slightly increase or decrease when $10^{-4}\leq \nu\leq 1$, and they stay the same as $\nu\rightarrow0$.
Thus, the errors of the \texttt{PR-EG} method are almost 100 times smaller than the \texttt{ST-EG} method when $\nu\leq 10^{-4}$, which means the \texttt{PR-EG} method solves the Brinkman equations with small $\nu$ more accurately.
The error profile functions show similar error behaviors in Figure~\ref{figure: errors_prof_3d}, supporting error estimates \eqref{sys: sum_err_esti_st} and \eqref{sys: sum_err_esti_ur}.

\subsubsection{Permeability test}

We apply piecewise constant permeability to the Brinkman equations \eqref{sys:governing} in the cube domain $\Omega=(0,1)^3$,
\begin{equation*}
    K(\bx) = \left\{\begin{array}{cl}
        10^{-6} & \text{if}\ |\bx|\leq (0.25)^2, \\
        1 & \text{otherwise}.
    \end{array}\right.
\end{equation*}
The other conditions are given as; viscosity $\mu=10^{-6}$, boundary condition $\bu=\langle 1,0,0\rangle$, and body force $\bbf=\langle 1, 1,1\rangle$.
\begin{figure}[!htb]
\centering
\begin{subfigure}{0.4\linewidth}
    \centering
    \includegraphics[width=1\linewidth]{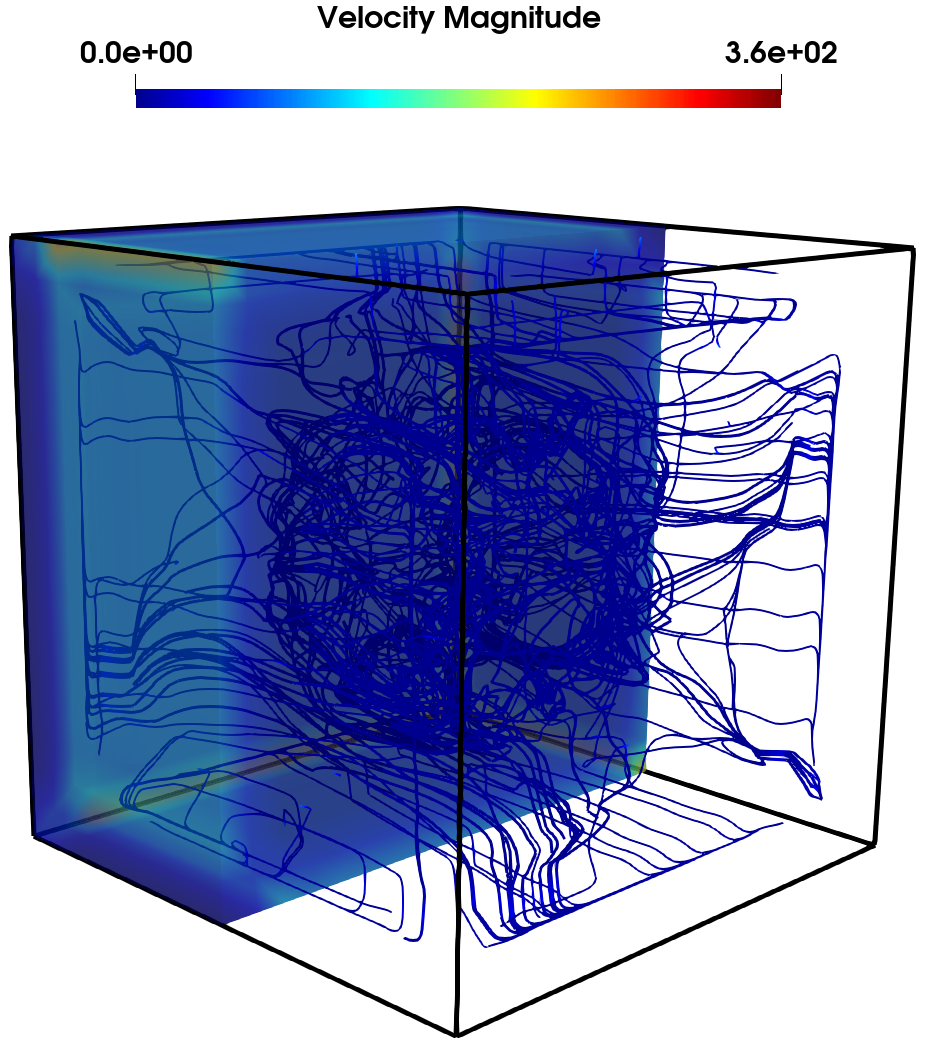}
    \caption{\texttt{ST-EG}: Streamlines and magnitude}
    \end{subfigure}
    \hskip 20pt
    \begin{subfigure}{0.4\linewidth}
    \centering
    \includegraphics[width=1\linewidth]{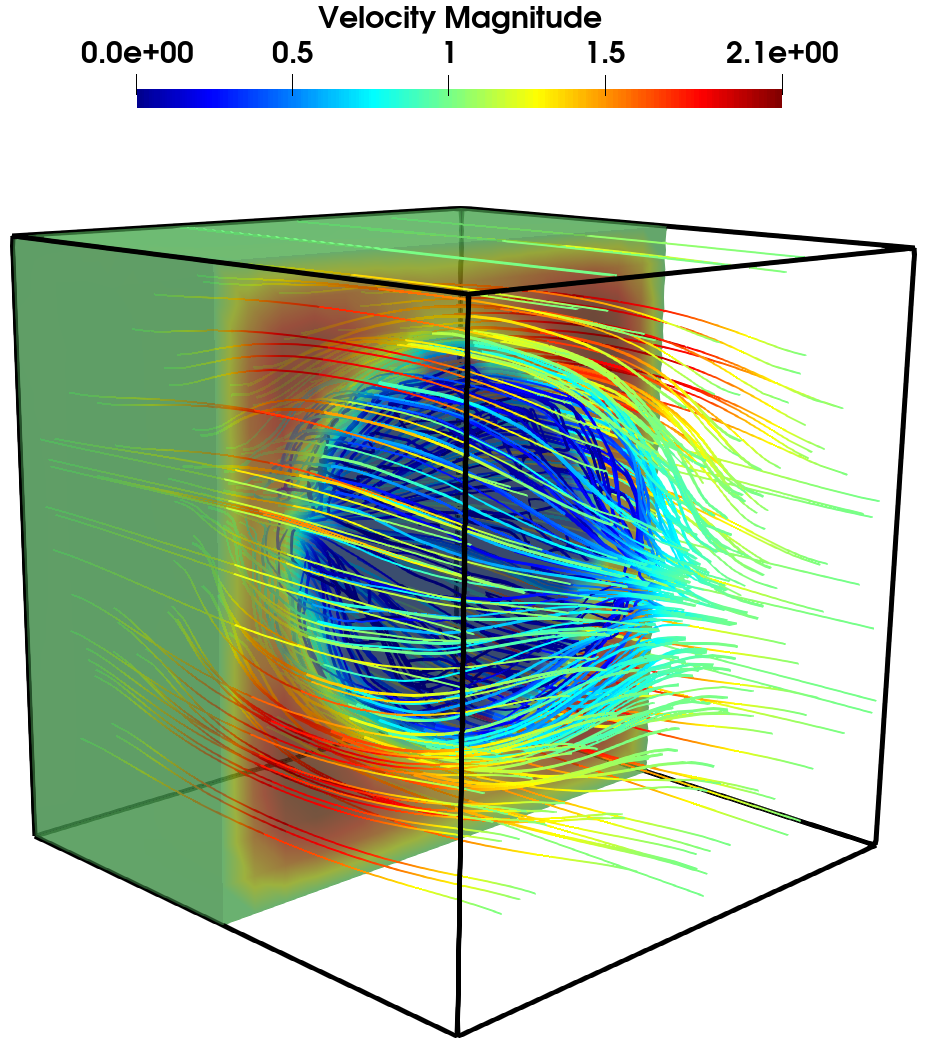}
        \caption{\texttt{PR-EG}: Streamlines and magnitude}
    \end{subfigure}
    \caption{Numerical velocity solutions of \texttt{ST-EG} and \texttt{PR-EG} when $h=1/16$.}
    \label{figure: permeability_3d}
\end{figure}
We expect the fluid flow to be faster out of the ball with small permeability, and it tends to avoid the ball and be affected by the boundary velocity.
The streamlines and colored magnitude of the \texttt{PR-EG} method's velocity in Figure~\ref{figure: permeability_3d} exactly show such an expectation on the fluid flow, while the \texttt{ST-EG} method fails to provide a reliable velocity solution.


\section{Conclusion}
\label{sec:conclusion}

In this paper, we proposed a pressure-robust numerical method for the Brinkman equations with minimal degrees of freedom based on the EG piecewise linear velocity and constant pressure spaces \cite{YiEtAl22Stokes}.
To derive the robust method, we used the velocity reconstruction operator \cite{HuLeeMuYi} mapping the EG velocity to the first-order Brezzi-Douglas-Marini space.
Then, we replaced the EG velocity in the Darcy term and the test function on the right-hand side with the reconstructed velocity. With this simple modification, the robust EG method showed uniform performance in both the Stokes and Darcy regimes compared to the standard EG method requiring the
mesh restriction $h<\sqrt{\nu}$ that is impractical in the Darcy regime.
We also validated the error estimates and performance of the standard and robust EG methods through several numerical tests with two- and three-dimensional examples.

Our efficient and robust EG method for the Brinkman equations can be extended to various Stokes-Darcy modeling problems, such as coupled models with an interface and time-dependent models. Also, 
the proposed EG method can be extended for nonlinear models, such as nonlinear Brinkman models for non-Newtonian fluid and unsteady Brinkman-Forchheimer models.

	\bibliographystyle{plain}
\bibliography{Brinkman}
	
\end{document}